\documentclass[letterpaper]{article} 
\usepackage{aaai25}  
\usepackage{times}  
\usepackage{helvet}  
\usepackage{courier}  
\usepackage[hyphens]{url}  
\usepackage{graphicx} 
\usepackage{natbib}  
\usepackage{caption} 
\usepackage{algorithm}
\usepackage{algorithmic}
\usepackage{subcaption}
\usepackage{float}
\usepackage{newfloat}
\usepackage{listings}
\DeclareCaptionStyle{ruled}{labelfont=normalfont,labelsep=colon,strut=off} 
\floatstyle{ruled}
\newfloat{listing}{tb}{lst}{}
\floatname{listing}{Listing}
\usepackage{amsmath}
\usepackage{amsthm}
\usepackage{booktabs}
\usepackage{algorithm}
\usepackage{algorithmic}
\usepackage[switch]{lineno}

\usepackage{chngcntr}

\usepackage{amssymb}
\usepackage{comment}
\usepackage{graphicx}
\usepackage[flushleft]{threeparttable}
\usepackage{multirow}
\usepackage{pifont}

\usepackage{xcolor} 
\definecolor{antiquegold}{RGB}{205,127,50} 
\definecolor{deepskyblue}{RGB}{0,191,255} 
\definecolor{crimson}{RGB}{220,20,60} 
\definecolor{orange}{RGB}{255,165,0} 
\definecolor{green}{RGB}{0,128,0}


\newcommand{\diag}{\mathop{\mathbf{diag}}}
\newcommand{\abs}[1]{\ensuremath{\left|#1\right|}}
\newcommand{\norm}[2][]{\ensuremath{\left\Vert #2 \right\Vert}}

\renewcommand{\vec}[1]{\mathbf{#1}}
\newcommand{\vect}[1]{\boldsymbol{\mathbf{#1}}}
\newcommand{\argmin}{\mathop{\mathbf{argmin}}}

\newcommand{\vol}{\mathrm{Vol}}
\newcommand{\grad}{\mathrm{grad}}

\newcommand{\Retr}{\mathrm{Retr}}

\newcommand{\Exp}{\mathrm{Exp}}

\newcommand{\Div}{\mathrm{div}}

     \newcommand{\BR}{{\mathbb {R}}}

 \newcommand{\cmark}{\ding{51}}%
\newcommand{\xmark}{\ding{55}}%

    \theoremstyle{plain}
    \newtheorem{thm}{Theorem}[section]

\theoremstyle{plain}
\newtheorem{theorem}{Theorem}[section]
\newtheorem{lemma}[theorem]{Lemma}
\newtheorem{proposition}[theorem]{Proposition}

\theoremstyle{definition}

\newtheorem{assumption}{Assumption}

\title{Langevin Multiplicative Weights Update \\ with Applications in Polynomial Portfolio Management}
\author{
    Yi Feng \equalcontrib \textsuperscript{\rm 1},
    Xiao Wang \equalcontrib \thanks{Correspondence to Xiao Wang.} \textsuperscript{\rm 1},
    Tian Xie \equalcontrib \textsuperscript{\rm 2} \\
}
\affiliations{
    \textsuperscript{\rm 1}School of Information Management \& Engineering, Shanghai University of Finance and Economics\\
    \textsuperscript{\rm 2}College of Business, Shanghai University of Finance and Economics\\
    fengyi95524@gmail.com, 
    wangxiao@sufe.edu.cn, 
    xietian@shufe.edu.cn

}

\usepackage{bibentry}

\begin{document}

\maketitle

\begin{abstract}
 We consider nonconvex optimization problem over simplex, and more generally, a product of simplices. We provide an algorithm, Langevin Multiplicative Weights Update (LMWU) for solving global optimization problems by adding a noise scaling with the non-Euclidean geometry in the simplex. Non-convex optimization has been extensively studied by machine learning community due to its application in various scenarios such as neural network approximation and finding Nash equilibrium. Despite recent progresses on provable guarantee of escaping and avoiding saddle point (convergence to local minima) and global convergence of Langevin gradient based method without constraints, the global optimization with constraints is less studied. We show that LMWU algorithm is provably convergent to interior global minima with a non-asymptotic convergence analysis. We verify the efficiency of the proposed algorithm in real data set from polynomial portfolio management, where optimization of a highly non-linear objective function plays a crucial role.
\end{abstract}

%

\section{Introduction}
In this paper we consider nonconvex optimization problem with constraint that is a product of simplices, i.e.,
\begin{equation}\label{setup}
\min_{\vec{x}\in\Delta_1\times...\times\Delta_N} f(\vec{x})
\end{equation}
where $f:\Delta_1\times...\times\Delta_N\rightarrow\mathbb{R}$ is a sufficiently smooth function and 
\[
\Delta_i=\left\{(x_{i1},...,x_{id}):\sum_{s=1}^dx_{is}=1,x_{is}\ge 0\right\}.
\]
Problem \eqref{setup} appears naturally in potential game \cite{Shapley}, i.e., the incentive of all players to change their strategy can be expressed using a single global function (the potential function). A natural approach is to use projected gradient descent, but computing the projection at every iteration might not be an easy task to accomplish. An alternate effective algorithm in solving problem \eqref{setup} is so called Multiplicative Weights Update (MWU) \cite{AHK12}, which is a special case of FTRL that is commonly used in min-max optimization and multi-agent systems \cite{Lei2021,FLPW24,FPW2024}. Result of \cite{ICMLPPW19} indicates that MWU almost always converges to second-order stationary points with random initialization. Besides MWU, many first-order methods have been proven escaping saddle points or avoiding saddle points asymptotically \cite{GHJY2015,Jin17,JNJ18,LSJR16,LP19,panageas2016gradient,CB19,SFF19,PPW19a,SLQ19}.

However, in the nonconvex world, finding local minima can be far away from achieving global minima. 
The classic MWU together with its accelerated variant \cite{FPW2022} can only converge to second-order stationary points or interior local minima, and this leaves finding global optima a challenging direction. One approach in designing first-order algorithm converging to global minima is to introduce a random noise into gradient descent, so that the algorithm has a chance to escape local minima. In recent years, progress has been made on this via the Langevin algorithm, an algorithm originally invented to sample from a target distribution proportional to $e^{-f(\vec{x})}$ where $f(\vec{x})$ is objective function defined on the whole Euclidean space $\mathbb{R}^n$. Successfully, global convergence of Langevin gradient descent with non-asymptotic convergence rate are obtained in \cite{RRT17,XCZG18}. More recently, projected Langevin algorithm has been investigated in \cite{Lamperski} from the perspective of constrained sampling and optimization.

Despite aforementioned progresses in local and global convergence of gradient based algorithms, it is less understood whether there exists an algorithm that naturally fits distributed optimization framework from game theory and multi-agent systems. It is indicated in \cite{BP2019} that projected gradient descent can spend a few steps at each corner if the constraint has lower dimension. This feature makes projected gradient descent in multi-agent systems less effective, and in contrast, MWU and its variants have proven prominent in learning of games, their behaviors have been extensively studied in literatures, e.g., \cite{PPP2017,BG18,Cheung,ChGP19,CP2020}. Nevertheless, finding global minima of potential games with MWU or any of its variant seems missing in literature.

\begin{table*}
\centering
\begin{tabular}{ |p{7.8cm}||p{2.2cm}|p{1.5cm}|p{1.5cm}| p{1.7cm} |}
 \multicolumn{4}{c}{}
   \\
 \hline
 & Global Convergence & Constraints& Simple Projection &Distributed Constraints\\
 \hline
MWU  \cite{ICMLPPW19}   & \xmark    &\cmark& \cmark  &\cmark\\
Langevin GD \cite{RRT17}&  \cmark   & \xmark  &\xmark &\xmark\\
Projected Langevin \cite{Lamperski} &\cmark & \cmark&  \xmark& \xmark\\
Perturbed RGD  \cite{CB19}   &\xmark & \cmark&  \cmark&\xmark\\
Accelerated MWU  \cite{FPW2022} &\xmark & \cmark & \cmark &\cmark\\
Langevin MWU (this work)&   \cmark  & \cmark&\cmark &\cmark\\
 \hline
\end{tabular}
\caption{Comparison to related results}
\label{tab:comparison} 
\end{table*}

Motivated by global convergence analysis of Langevin gradient descent algorithm \cite{RRT17}, we propose a scheme of adding noise that is scalable with a natural geometry of simplex, so that the Langevin Multiplicative Weights Update algorithm (LMWU) enjoys both the efficiency of projecting onto the constraint and the power of escaping saddle points and spurious local minima. LMWU is derived from the geometric Brownian motion on Riemannian manifold, where the natural geometry of simplex, i.e., Shahshahani geometry, plays a crucial role. The main result is stated as follows, and our contributions compared to the most relevant results in literature are illustrated in above table.

\begin{thm}[Informal]
Suppose the global optima of Problem \eqref{setup} is in the interior of the constraints. The Langevin Multiplicative Weights Update converges to the biased global optima in expectation.
\end{thm} 

\paragraph{Other related works.}There have been considerably amount of works in convergence to local and global optima with first-order methods. Apart from the references listed in Table 1, we give a relatively complete review on the literatures about local and global convergence with gradient descent and Langevin algorithms. Local convergence guarantee with non-asymptotic convergence rate are investigated in \cite{GHJY2015,Jin17,JNJ18,SFF19,He2024}.Asymptotic convergence to local optima is studied with techniques from dynamical systems, typical references include \cite{LP19,LSJR16,AMPW2022}. On the other hand, convergence of Langevin algorithm in sampling has attracted many attentions. When the target distribution is log-concave, Euler discretization converges rapidly \cite{RT96,Dal17}. Later on the convergence rate was improved in \cite{DM17}. More recently, rapid convergence of Langevin algorithm for distributions satisfying log-Sobolev inequality has been established in \cite{VW19,LE20,WLP20,GV2022}. An improved rate analysis for Langevin SGD with variance reduction is provided in \cite{KS22} For sampling in a constrained set, Mirror Langevin diffusion has been studied in \cite{ZPFP20,HKRC18,AC21,Jia21,LTVW21}. A Reflected Langevin algorithm is proposed and analyzed in \cite{STKS22}, we need to mention that the reflected operation has an projection operation embedded, which makes it difficulty to apply the algorithm in simplicial constraint. 

\section{Preliminaries}
This section reviews the main background on Riemannian geometry and probability distributions on manifolds.
\subsection{Riemannian Geometry}
\paragraph{Riemannian metric and exponential map.}A Riemannian manifold $(M,g)$ is real, smooth manifold $M$ equipped with a Riemannian metric $g$. For each $\vec{x}\in M$, let $T_{\vec{x}}M$ denote the tangent space at $\vec{x}$. The metric $g$ induces a inner product $\langle\cdot,\cdot\rangle_{\vec{x}}:T_{\vec{x}}M\times T_{\vec{x}}M\rightarrow\mathbb{R}$. We call a curve $\gamma(t):[0,1]\rightarrow M$ a geodesic if it satisfies 
\begin{itemize}
\item The curve $\gamma(t)$ is parametrized with constant speed, i.e. $\norm{\frac{d}{dt}\gamma(t)}_{\gamma(t)}$ is constant for $t\in[0,1]$.
\item The curve is locally length minimized between $\gamma(0)$ and $\gamma(1)$.
\end{itemize}

\paragraph{Riemannian gradient.}
 For differentiable function $f:M\rightarrow\BR$, $\grad f(\vec{x})\in T_{\vec{x}}M$ denotes the Riemannian gradient of $f$ that satisfies $\frac{d}{dt}f(\gamma(t))=\langle\gamma'(t),\grad f(\vec{x})\rangle$ for any differentiable curve $\gamma(t)$ passing through $\vec{x}$. The local coordinate expression of gradient is useful in our analysis.
\begin{equation}\label{grad}
\grad f(\vec{x})=\left(\sum_jg^{1j}(\vec{x})\frac{\partial f}{\partial x_j},...,\sum_jg^{dj}(\vec{x})\frac{\partial f}{\partial x_j}\right)
\end{equation}
where $g^{ij}(\vec{x})$ is the $ij$-th entry of the inverse of the metric matrix $\{g_{ij}(\vec{x})\}$ at each point. 

\paragraph{Retraction.} A retraction on a manifold $M$ is a smooth mapping $\Retr$ from the tangent bundle $TM$ to $M$ satisfying properties 1 and 2 below: Let $\Retr_{\vec{x}}:T_{\vec{x}}M\rightarrow M$ denote the restriction of $Retr$ to $T_{\vec{x}}M$.
\begin{enumerate}
\item $\Retr_{\vec{x}}(0)=\vec{x}$, where $0$ is the zero vector in $T_{\vec{x}}M$.
\item The differential of $\Retr_{\vec{x}}$ at $0$ is the identity map.
\end{enumerate}
Then the Riemannian gradient descent with stepsize $\alpha $ is given as
\begin{equation}\label{GD:Riemannian}
\vec{x}_{t+1}=\Retr_{\vec{x}_t}(-\epsilon\grad f(\vec{x}_t)).
\end{equation}

\subsection{Distributions on manifold}
\paragraph{KL divergence.}Let $\rho$ and $\nu$ be probability distributions on $M$ that is absolutely continuous with respect to the Riemannian volume measure on $M$ (denoted as $d\vec{x}$). The \emph{Kullback-Leibler} (KL) divergence of $\rho$ with respect to $\nu$ is 
\[
H(\rho|\nu)=\int_M\rho(\vec{x})\log\frac{\rho(\vec{x})}{\nu(\vec{x})}d\vec{x}
\]
KL-divergence measures the ``distance" between two probability distributions. Note that KL-divergene is nonnegative: $H(\rho|\nu)\ge 0$, and it is minimized at the target distribution, i.e., $H(\rho|\nu)=0$ if and only if $\rho=\nu$. Furthermore, $\nu$ is the only stationary point of $H(\cdot|\nu)$, and thus sampling from $\nu$ can be reduced to minimizing $H(\cdot|\nu)$. Note that if $\nu=e^{-\beta f}$, the KL-divergence can be decomposed into 
\[
H(\rho|\nu)=\mathbb{E}_{\rho}f+\mathcal{H}(\rho),
\]
where $\mathbb{E}_{\rho}f=\int_M\rho fd\vol$ is the expected value of $f$ and $\mathcal{H}(\rho)=-\int_M\rho\log\rho d\vol$ is the differential entropy of $\rho$.
\paragraph{Wasserstein distance.}The Wasserstein distance between $\mu$ and $\nu$ is defined to be
\[
\inf\{\sqrt{\mathbb{E}[d(X,Y)^2]}:\text{law}(X)=\mu,\text{law}(Y)=\nu\}
\]

\paragraph{Log-Sobolev inequality.} A probability measure $\mu$ on $M$ is called to satisfy the logarithmic Sobolev inequality if there exists a constant $\alpha>0$ such that

\begin{align}
&\int_Mg^2\log g^2d\nu-\left(\int_Mg^2d\nu\right)\log\left(\int_Mg^2d\nu\right)\notag\\
&\le\frac{2}{\alpha}\int_M\norm{\grad g}^2d\nu
\end{align}

for all smooth functions $g:M\rightarrow\mathbb{R}$ with $\int_Mg^2\le\infty$. 
The relative Fisher information of $\rho$ with respect to $\nu$ is $I_{\nu}(\rho)=\int_M\rho(\vec{x})\norm{\grad\log\frac{\rho(\vec{x})}{\nu(\vec{x})}}d\vol$. Log-Sobolev inequality (LSI) is equivalent to the relation between KL-divergence and Fisher information: $H(\rho|\nu)\le\frac{1}{2\alpha}I_{\nu}(\rho)$.

\section{Main Results}
In this section, we review classic Multiplicative Weights Update and its linear variant, and then some well known facts about Shahshahani geometry will be discussed. Based on the geometric setting of the simplex, we give a sketched framework how the Langevin Multiplicative Weights Update is derived.

\subsection{From MWU to Langevin MWU} 
 The classic Multiplicative Weights Update  is widely used in constrained optimization, multi-agent system and game theory. It often refers to two forms, 
\[
x_{ij}(k+1)=\frac{x_{ij}(k)e^{-\epsilon\frac{\partial f}{\partial x_{ij}}}}{\sum_s x_{is}(k)e^{-\epsilon\frac{\partial f}{\partial x_{is}}}}
\]
and its linear variant.
If not specified, This paper refers MWU to the linear variant. For completeness, we recall the linear variant of MWU. Suppose that $\vec{x}_i=(x_{i1},...,x_{id_i})$ is in the $i$-th component of $\Delta_1\times...\times\Delta_n$. Assume that $\vec{x}(k)$ is the $k$-th iterate of MWU, the algorithm is written as follows:
\begin{equation}\label{MWUclassic}
x_{ij}(k+1)=x_{ij}(k)\frac{1-\epsilon \frac{\partial f}{\partial x_{ij}}}{1-\epsilon\sum_{s}x_{is}(k)\frac{\partial f}{\partial x_{is}}},
\end{equation}
where $j\in\{1,...,d_i\}$.

\begin{algorithm}[t]
\caption{Langevin-MWU (single-agent) }
\label{alg:C}
\begin{algorithmic}
\STATE Input : error threshold $\delta>0$, large enough $\beta>0$,
\\
Compute step size $\epsilon<\frac{\delta^2\alpha}{8C(\frac{M}{2}\sigma+B)}$,
\\Initialize $\vec{x}_0\sim \rho_0$, 
\\
\REPEAT
\STATE Compute $S_{\vec{x}}=\sum_{j=1}^n\frac{1}{x_j}$, and $z_0^i\sim\mathcal{N}(0,1)$.
\\
Compute \\
$V_0^i=\frac{\epsilon}{2\beta}\left(n+1-(1+x_i)S_{\vec{x}}\right)+\sqrt{2\epsilon\beta^{-1} x_i}z_0^i$
\\
Set $x_i\leftarrow\frac{x_i-\epsilon x_i\frac{\partial f}{\partial x_i}+V_0^i}{1-\epsilon\sum_{j=1}^nx_j\frac{\partial f}{\partial x_j}+\sum_{j=1}^nV_0^j}$
\UNTIL{$k$ large enough, e.g., $k>\frac{16}{3\epsilon}\left(\frac{16(\frac{M}{2}\sigma+B)^2}{\delta_2\alpha}\right)$ }
\end{algorithmic}
\end{algorithm}

It is well known that Langevin dynamics corresponds to the gradient flow of relative entropy respect to Wasserstein metric. In the space of measures with the Wasserstein metric, the gradient flow of relative entropy is the following partial differential equation, called Entropy Regularized Wasserstein Gradient Flow:
\[
\frac{\partial \rho}{\partial t}=\nabla \cdot\left(\rho\nabla f\right)+\beta^{-1}\Delta\rho
\]
The key step in deriving Langevin Multiplicative Weights Update is to implement or approximate the noise scaled with the Shahshahani geometry in the simplex, which is a discretization of geometric Brownian motion in Shahshahani manifold. The geometric Brownian motion inside of the simplex $\Delta_+^{d-1}\subset\mathbb{R}_+^d$ can be obtained from the orthogonal projection of the geometric Brownian motion in $\mathbb{R}_+^d$, where the orthogonal projection is with respect to the Shahshahani metric in $\mathbb{R}_+^d$. Recall the standard Brownian motion in $\mathbb{R}^d$ is a random process $\{X_t\}_{t\ge 0}$ whose density function $\rho(\vec{x},t)$ evolves according to the diffusion equation 
\[
\frac{\partial\rho(\vec{x},t)}{\partial t}=\beta^{-1}\Delta\rho(\vec{x},t).
\]
The Brownian motion in Shahshahani manifold $\mathbb{R}_+^d$ is a random process $\{W_t\}_{t\ge 0}$ whose density function evolves according to the diffusion equation with respect to the Laplace-Beltrami operator, i.e.,
\[
\frac{\rho(\vec{x},t)}{\partial t}=\beta^{-1}\Delta_M\rho(\vec{x},t).
\]
Since $\mathbb{R}_+^d$ serves as its own local coordinate system as a Riemannian manifold, the geometric Brownian motion in $\mathbb{R}_+^d$ is described by the following stochastic differential equation 
\[
dX_t=-\beta^{-1}g^{ij}\Gamma_{ij}^kdt+\sqrt{2\beta^{-1}g^{-1}}dB_t
\]
where $dB_t$ is the standard Brownian motion in Euclidean space, $g^{ij}$ is the $(ij)$-entry of the inverse matrix of Shahshahani metric matrix $g_{ij}$, and $\Gamma_{ij}^k$ is the Christoffel symbol of Shahshahani metric that can be calculated explicitly. After establishing the noise discretized from Shahshahani geometric Brownian, we combine the noise and the Riemannian gradient in $\mathbb{R}_+^d$ to finalize the incremental vector in the update rule. We leave the details in Appendix. 

\subsection{Main Theorem}

In this section, we firstly state our main theorem that asserts the convergence in expectation of the L-MWU algorithm. Secondly, we will sketch the proof strategies, i.e., decomposition of error $\mathbb{E}f(\vec{x}_k)-f^*$ into 
\begin{equation}\label{decomp:error}
\mathbb{E}f(\vec{x}_k)-\mathbb{E}_{\nu}f+\mathbb{E}_{\nu}f-f^*
\end{equation}
where the expectation $\mathbb{E}_{\nu}f=\int_Mf(\vec{x})\nu(\vec{x}) d\vol$ and $f^*$ is the global minimum of $f(\vec{x})$ over $M$ and $\nu(\vec{x})$ is the probability density function that is proportional to $e^{-\beta f}$. We start presenting the main theorem by some a brief discussion on assumptions used in theoretical analysis.

Our analysis relies heavily on the theory of global convergence for Langevin algorithm in Euclidean space \cite{RRT17,XCZG18} and the results of rapid convergence results for log-Sobolev distributions such as \cite{GV2022}. Our strategy of giving theoretical analysis is to relate the assumptions in \cite{RRT17} to  the case of Shahshahani manifold, and then generalize the arguments in Euclidean space to Riemannian manifold with special structure. The reason that one can generalize the results in Euclidean space to Shahshahani manifold is the possibility of geometrizing the analytic assumption on $f$ by identifying $\vec{0}$ in $\mathbb{R}^n$ with $\frac{1}{n}(1,...,1)$ in $\Delta^{n-1}$, and the interior of $\Delta^{n-1}$ is diffeomorphic to $\mathbb{R}^{n-1}$. We start by giving assumptions function $f$ satisfies.

\begin{assumption}\label{A1}
 The function $f$ takes nonnegative real values, and there exist constants $A,B\ge 0$, such that
\[
\abs{f\left(\vec{1}\right)}\le A \ \ \text{and}\ \ \norm{\grad f\left(\vec{1}\right)}\le B.
\]
\end{assumption}
This assumption comes from assumption (A.1) in \cite{RRT17} by relating $\vec{0}$ to $\vec{1}=\frac{1}{n}(1,...,1)$.
\begin{assumption}\label{A2}
 Function $f$ is $M$-smooth for some $M>0$, i.e.,
 \[
 \norm{\grad f(\vec{y})-\Gamma_{\vec{x}}^{\vec{y}}f(\vec{x})}\le Md(\vec{x},\vec{y}) \ \ \text{for all}\ \ \vec{x},\vec{y}\in \mathcal{M},
 \]
 where $\Gamma_{\vec{x}}^{\vec{y}}$ denotes the parallel transport from $\vec{x}$ to $\vec{y}$. The gradient satisfies 
 \[
\norm{ \grad f(\vec{x})}_{\vec{x}}\le \frac{M}{2}d(\vec{1},\vec{x})+B
 \]
 for some constants $M>0$ and $B>0$.
 \end{assumption}
 $M$-smoothness in Euclidean setting reads as $\norm{\nabla f(\vec{x})-f(\vec{y})}\le M\norm{\vec{x}-\vec{y}}$, which is commonly assumed in many theoretical analysis. 
 
 \begin{assumption}\label{A3}
There exist positive numbers $m$ and $b$ such that
\[
\langle\grad f(\vec{x}),d(\vec{1},\vec{x})\vec{v}\rangle_{\vec{x}}\ge md(\vec{1},\vec{x})^2-b.
\]
where $\vec{v}$ is the velocity vector of the geodesic connecting $\vec{1}$ and $\vec{x}$. 
\end{assumption}
By a constant speed geodesic we mean the velocity has unit length everywhere, so the term $d(\vec{1},\vec{x})\vec{v}$ can be reduced to $\vec{x}$ in Euclidean space, where $\vec{x}$ means the geodesic of length $\norm{\vec{x}}$ (straight line) connecting $\vec{0}$ and $\vec{x}$.

\begin{assumption}\label{A4}
The differential entropy of the distribution $e^{-\beta f}$ is bounded by a constant $K$.
\end{assumption}
In the case of Euclidean space, the differential entropy has an upper bound by estimating the second moment of Gibbs distribution \cite{RRT17}. The differential entropy of a probability density with a finite second moment is upper-bounded by that of a Gaussian density with the same second moment, $h(\nu)\le\frac{d}{2}\log\left(\frac{2\pi e(b+d/\beta)}{md}\right)$. Thus there exists an upper-bound for $\beta$ large enough.
\begin{assumption}\label{A5}
$e^{-\beta f}$ satisfies log-Sobolev inequality.
\end{assumption}
This condition is necessary in bounding the sampling algorithm converges rapidly.

\begin{thm}
Suppose $f$ satisfies Assumptions \ref{A1}-\ref{A5}. Then there exists constant $C$, such that
\begin{align}
\abs{\mathbb{E}f(\vec{x}_k)-f^*}\le &(\frac{M}{2}\sigma+B)\sqrt{\frac{2}{\alpha}}\left(e^{-\frac{3}{16}\alpha\epsilon k}+\frac{C\epsilon}{\alpha}\right)^{\frac{1}{2}} \notag\\
&+\frac{K}{\beta}+\frac{1}{\beta}\log\left(\text{poly}(\beta^{-1})^{-1}\right).
\end{align}
\end{thm}

From the theorem we can conclude that for any given $\delta>0$, there exists $\beta>0$, $\epsilon<\frac{\delta^2\alpha}{8C(\frac{M}{2}\sigma+B)^2}$, and $k>\frac{16}{3\epsilon}\log\left(\frac{16(\frac{M}{2}\sigma+B)^2}{\delta^2\alpha}\right)$, such that $\abs{\mathbb{E}f(\vec{x}_k)-f^*}\le\delta$.

\subsection{Outline of Proof}
Suppose that the $k$'th iteration $\vec{x}_k$, which is a random variable on Shahshahani manifold $M$, has probability density function $\rho_k(\vec{x})$. Then the expectation $\mathbb{E}f(\vec{x}_k)$ can be written as 
\[
\int_Mf(\vec{x})\rho_k(\vec{x})d\vol
\]
where $d\vol$ is the Riemannian volume element induced by Shahshahani metric on $M$. Since the error $\mathbb{E}f(\vec{x}_k)-f^*$ has been decomposed into the sum of \eqref{decomp:error}, we need to bound $\abs{\mathbb{E}f(\vec{x}_k)-\mathbb{E}_{\nu}f}$ and $\abs{\mathbb{E}_{\nu}f-f^*}$ respectively. By the integral on manifold we have the following:
\begin{align}
\abs{\mathbb{E}f(\vec{x}_k)-\mathbb{E}_{\nu}f}&=\abs{\int_Mf(\vec{x})\rho_k(\vec{x})-\int_Mf(\vec{x})\nu(\vec{x})}\notag \\
&=\abs{\int_Mf(\vec{x})(\rho_k(\vec{x})-\nu(\vec{x}))d\vol}
\end{align}
 In Euclidean space, the difference is bounded by the Wasserstein distance between $\rho_k$ and $\nu$ according to Lemma 6 of \cite{RRT17}, where the authors prove that $\abs{\int_{\mathbb{R}^d}gd\mu-\int_{\mathbb{R}^d}gd\nu}\le\text{const}\cdot W_2(\mu,\nu)$, if $g$, $\mu$ and $\nu$ satisfy some assumptions. Therefore, our strategy of bounding $\abs{\int_Mf\rho_kd\vol-\int_Mf\nu d\vol}$ relies on a generalized version of Lemma 6 of \cite{RRT17} in the case of Shahshahani manifold, if not for all Riemannian manifolds. Following this idea, we provide the following lemma.
 
\begin{lemma}\label{lemma:convergence1}
Let $\mu$ and $\nu$ be two density function of probability measures on Shahshahani manifold $M$. Suppose $f:M\rightarrow\mathbb{R}$ satisfies 
\[
\norm{\grad f(\vec{x})}\le \frac{M}{2}d(\vect{1},\vec{x})+B
\]
for some constants $\frac{M}{2}>0$ and $B>0$. Then
\[
\abs{\int_Mf\mu d\vol-\int_M f\nu d\vol}\le(\frac{M}{2}\sigma+B)W_2(\mu,\nu)
\]
where $\sigma^2=\int_Md(\vec{1},\vec{x})^2\mu(\vec{x})d\vol\vee\int_Md(\vec{1},\vec{y})^2\nu(\vec{y})d\vol$.
\end{lemma}
Letting $\mu=\rho_k$, we can immediately obtain the expected result, i.e.,
\begin{align}
\abs{\mathbb{E}f(\vec{x}_k)-\mathbb{E}_{\nu}f}&=\abs{\int_Mf\rho_kd\vol-\int_Mf\nu d\vol}\notag\\
&\le(\frac{M}{2}\sigma+B)W_2(\rho_k,\nu)
\end{align}
Talagrand inequality is a well known connection between Wasserstein distance and KL-divergence. We say that a probability measure $\nu$ satisfies a Talagrand inequality with constant $\alpha>0$ if for all probability measure $\rho$, absolutely continuous with respect to $\nu$, with finite moments of order 2, it holds that
$
W_2(\rho,\nu)^2\le\frac{2}{\alpha}H(\rho|\nu).
$ 
Therefore, bounding $W_2(\rho_k,\nu)$ boils down to bounding the KL-divergence $H(\rho_k|\nu)$. It has been shown in \cite{GV2022} that for general Hessian Manifold, Langevin algorithm for sampling from a log-Sobolev distribution converges rapidly to a distribution with bias $\epsilon$. Since simplex with Shahshahani metric is a Hessian manifold, applying Theorem 2 of \cite{GV2022}, we can immediately conclude that there exists a constant $C$ and log-Sobolev constant $\alpha$ such that $H(\rho_k|\nu)\le e^{-\frac{3}{16}\alpha\epsilon k}+\frac{C\epsilon}{\alpha}$, and therefore
\begin{align}
\abs{\mathbb{E}f(\vec{x}_k)-\mathbb{E}_{\nu}f}&\le(\frac{M}{2}\sigma+B)\sqrt{\frac{2}{\alpha}}H(\rho_k|\nu)^{\frac{1}{2}}\notag\\
&\le(\frac{M}{2}\sigma+B)\sqrt{\frac{2}{\alpha}}\left(e^{-\frac{3}{16}\alpha\epsilon k}+\frac{C\epsilon}{\alpha}\right)^{\frac{1}{2}}.
\end{align}

To see that $\phi=\sum_{i=1}^nx_i\ln x_i$ induces a Hessian metric on simplex, let $x_n=1-\sum_{i=1}^{n-1}x_i$, then 
\[
\phi=\left(1-\sum_{i=1}^{n-1}x_i\right)\ln\left(1-\sum_{i=1}^{n-1}x_i\right).
\]
The Hessian $\nabla^2\phi$ has the form of the following:

\begin{equation}\label{eq:Hess}
\left[
\begin{array}{ccc}
\frac{1}{x_1}+\frac{1}{1-\sum_{i=1}^{n-1} x_i}&\dots&\frac{1}{1-\sum_{i=1}^{n-1}x_i}
\\
\\
\vdots&\ddots&\vdots
\\
\\
\frac{1}{1-\sum_{i=1}^{n-1}x_i}&\dots&\frac{1}{x_{n-1}}+\frac{1}{1-\sum_{i=1}^{n-1}x_i}
\end{array}
\right].
\end{equation}

On the other hand, the mapping $\varphi:(x_1,...,x_{n-1})\rightarrow(x_1,...,x_{n-1},1-\sum_{i=1}^{n-1}x_i)$ from $\mathbb{R}^{n-1}$ to $\mathbb{R}^n$ induces a Riemannian metric in the projection of simplex, and this metric matrix $\langle d\varphi(\cdot),d\varphi(\cdot)\rangle$ is exactly the same as \eqref{eq:Hess}.

Running Langevin dynamics is equivalent to optimization in the space of probability densities in the underlying space \cite{Wibisono18}, and thus equivalent to sampling from the stationary distribution of the Wasserstein gradient flow asymptotically. To minimize $\int_Mf(\vec{x})\rho(\vec{x})d\vec{x}$ with respect to $\rho(\vec{x})$, we introduce the entropy regularized functional of $\rho$ defined by
$
\mathcal{L}(\rho)=\mathcal{F}(\rho)+\beta^{-1}\mathcal{H}(\rho)
$
where 
$
\mathcal{F}(\rho)=\int_Mf(\vec{x})\rho(\vec{x})d\vec{x},
$ 
and 
$
\mathcal{H}(\rho)=-\int_M\rho(\vec{x})\log\rho(\vec{x})d\vec{x}.
$
The Wasserstein space $\mathcal{P}_2(\mathcal{M})$ of probability measures on $\mathcal{M}$ is an infinite dimensional smooth Riemannian manifold. A tangent vector $R\in T_{\rho}\mathcal{M}$ is of the form $R=-\Div\left(\rho\grad \phi\right)$ for some function $\phi:\mathcal{M}\rightarrow\mathbb{R}$. The gradient of a functional $\mathcal{L}:\mathcal{P}\rightarrow\mathbb{R}$ is $\grad_{\rho}\mathcal{L}=-\Div\left(\rho\grad\frac{\delta\mathcal{L}}{\delta\rho}\right)$, where $\frac{\delta\mathcal{L}}{\delta\rho}(\vec{x})$ is the first variation of $\mathcal{L}$ with respect to $\rho$.
 It is well known that the Wasserstein gradient flow of $\mathcal{L}$ is the Fokker-Planck equation
\begin{align}\label{FP}
\frac{\partial \rho(\vec{x},t)}{\partial t}&=\Div(\rho(\vec{x},t)\grad f(\vec{x})+\beta^{-1}\grad\rho(\vec{x},t))\notag\\
&=\Div(\rho(\vec{x},t)\grad f(\vec{x}))+\beta^{-1}\Delta_M\rho(\vec{x},t),
\end{align}
where $\grad$, $\Div$ and $\Delta_M$ are gradient, divergence and Laplace-Beltrami on manifolds. The stationary solution of equation \eqref{FP} is the density proportional to $e^{-\beta f}$ that minimizes $\mathcal{L}$.
\begin{lemma}\label{lemma:convergence2}
Suppose the entropy of distribution $\nu(\vec{x})$ is uniformly bounded for all $\beta$, i.e., $h(\nu)\le K<\infty$. Then
\[
\abs{\mathbb{E}_{\nu}f-f^*}\le\frac{K}{\beta}+\frac{1}{\beta}\log\left(\text{poly}\left(\frac{1}{\beta}\right)^{-1}\right).
\]
\end{lemma}
Let $p(\vec{x})=\frac{e^{-\beta f(\vec{x})}}{\Lambda}$ denote the density of the Gibbs measure with respect to the measure induced by the Shahshahani metric in simplex, where
$
\Lambda:=\int_Me^{-\beta f(\vec{x})}d\vec{x}
$
is the normalization constant known as the partition function. Note that the differential entropy of $p$ has the following expression,
\[
h(p)=\frac{1}{\Lambda}\int_M\beta f(\vec{x})e^{-\beta f(\vec{x})}d\vol+\log\Lambda
\]
thus we have that
\[
\int_Mf(\vec{x})p(\vec{x})d\vol=\frac{1}{\beta}(h(p)-\log\Lambda).
\]
Let $\vec{x}^*$ be any point that minimizes $f(\vec{x})$. Then $\grad f(\vec{x}^*)=0$. Since $f$ is assumed to be geodesically smooth, we have $f(\vec{x})-f(\vec{x}^*)\le\frac{M}{2}d(\vec{x},\vec{x}^*)^2$, the lower bound of $\log\Lambda$ can be obtained by following calculation,
\begin{align}
\log \Lambda&=\log\int_Me^{-\beta f(\vec{x})}d\vol\notag\\
&=-\beta f(\vec{x}^*)+\log\int_Me^{-\beta(f(\vec{x}^*)-f(\vec{x}))}d\vol
\notag\\
&\ge -\beta f(\vec{x}^*)+\log\int_Me^{-\beta d(\vec{x},\vec{x}^*)^2/2}d\vol
\end{align}
Without loss of generality, we can assume that the global minima $\vec{x}^*$ is at the center of simplex, i.e., $\vec{x}^*=\vec{1}=\left(\frac{1}{n},...,\frac{1}{n}\right)$. In appendix, we show that the integral $\int_Me^{-cd(\vec{1},\vec{x})^2}d\vol$ is bounded. By letting $c=\beta\frac{M}{2}$, we furthermore end up with a concrete expression of $\int_Me^{-cd(\vec{1},\vec{x})^2}d\vol$ in terms of a polynomial of $\beta^{-1}$, which is denoted briefly as follows,
\[
\log\Lambda\ge -\beta f(\vec{x}^*)+\log\left(\text{poly}\left(\beta^{-1}\right)\right),
\]
and then we have
\[
-f(\vec{x}^*)\le\frac{\log\Lambda}{\beta}+\frac{1}{\beta}\log\left(\text{poly}\left(\beta^{-1}\right)^{-1}\right).
\]
Combining with $\mathbb{E}_{\nu}f=\frac{h(\nu)}{\beta}-\frac{\log\Lambda}{\beta}$, we have the following bound:
\[
\mathbb{E}_{\nu}f-f(\vec{x}^*)\le\frac{K}{\beta}+\frac{1}{\beta}\log\left(\text{poly}\left(\beta^{-1}\right)^{-1}\right).
\]

\section{Application in Portfolio Management}


\begin{table*}[htb] 
\centering 
\begin{threeparttable} 
\begin{tabular}{lccccccc} 
\hline 
Method & Degenerate & Increasing & MV & MVS & MVSK & Equal\\
\hline
MWU   &  74.7203 &  17.4527 &   0.8561 &   0.8391 &  16.1104 &  44.5159  \\ 
AMWU  \cite{FPW2022}&  76.3657 &  17.5554 &   0.8561 &   0.8579 &  15.6559 &  43.4190  \\ 
Projected Langevin  \cite{Lamperski} &  73.9596 &  17.8720 &   0.8674 &   0.8846 &  16.7519 &  43.2847  \\ 
LMWU (this work)&  70.8930 &  16.7684 &   0.8464 &   0.8314 &  14.9585 &  42.8307  \\ 
\hline 
\end{tabular} 
\end{threeparttable} 
\caption{Out-of-sample Evaluation Results for Polynomial Portfolio Optimization} \label{result} 
\end{table*} 

 Portfolio management is a critical aspect of finance as it facilitates the efficient and effective management of investments to achieve specific financial goals and objectives. It involves the careful selection, diversification, and alignment of various financial instruments such as stocks, bonds, and other assets, to balance risk and returns according to an individual or institution's risk tolerance, time horizon, and investment objectives. The strategic allocation of assets in a portfolio can enhance returns, mitigate potential losses, and provide a smoother investment journey. Moreover, portfolio management offers a structured approach to monitor, review, and adjust investments in response to changing market conditions, personal circumstances, or shifts in financial goals, making it an indispensable tool for successful financial planning and wealth management.
 
The polynomial portfolio optimization problem can be formally represented as
\[
\hat{\vec{ w}} = \underset{\vec{ w}\in\mathcal {W}}{\argmin} \, \mathbb{E}[f(\vec{ w}, \vec{ r})], \label{fun}
\]
where $\mathbb{E}$ denotes the expectation operator, $f(\vec{w}, \vec{r})$ refers to a polynomial loss function, and $\vec{r} = [r_1, r_2, ..., r_n]^\top$ symbolizes the vector of $n$ individual returns within the portfolio. Additionally, $\vec{w} = [w_1, w_2, ..., w_n]^\top$ signifies the weights designated to each constituent element of the portfolio. It's important to note that $\vec{w}$ is restricted to the feasible set $\mathcal{W}$,
$
\mathcal{W} \equiv \{\vec{ w}\in\mathbb {R}^N: \sum^N_{i=1} w_i = 1 \},$ which constrains the summation of the weights to be one. This constraint implies that no leveraging or borrowing is permitted in the portfolio construction.

We propose a specific formulation for the loss function $f(\vec{w}, \vec{r})$ as follows:
\begin{align}
f(\vec{ w}, \vec{ r}) = -\lambda_1m_1(\vec{ w}, \vec{ r})& + \lambda_2m_2(\vec{ w}, \vec{ r}) + ... \notag\\
&+ (-1)^{d} \lambda_dm_d(\vec{ w}, \vec{ r}), \label{lossfun}
\end{align}
where $m_1(\vec{w}, \vec{r}) = \vec{w}^{\top} \vec{r}$ represents the sample portfolio return, and
\[
m_i(\vec{ w}, \vec{ r}) = \Big(m_1(\vec{ w}, \vec{ r}) - \mathbb{E}\big( m_1(\vec{ w}, \vec{ r})\big)\Big)^i, \quad i = 2,...,d
\]
encapsulates the $i^{th}$ central moment of $m_1(\vec{w}, \vec{r})$, with $\mathbb{E}\big(m_i(\vec{w}, \vec{r})\big)$ being the expected value. The parameter vector $\vec{\lambda} = [\lambda_1,...,\lambda_d]^\top$ contains the risk preference parameters, each satisfying $\lambda_i\ge0$, and their summation amounts to one, i.e., $\sum^d_{i=1}\lambda_i = 1$. It's worth noting that the mean-variance (MV), mean-variance-skewness (MVS), and mean-variance-skewness-kurtosis (MVSK) losses can be considered specific instances of this general polynomial portfolio optimization framework.

Our dataset comprises daily entries for $n = 10$ notable NASDAQ stocks, covering the period from January 3, 2011, to December 31, 2021, and thereby accumulating $T = 2517$ periods. We initiate a rolling-window out-of-sample forecasting exercise from the beginning of this data sample. The window length is set at $L=1000$, approximately corresponding to four years of training data. To calculate the optimal portfolio weights, $\hat{\vec{w}}$, we implement four estimation strategies: the traditional Multiplicative Weight Update (MWU) approach, the accelerated MWU algorithm purposed in \cite{FPW2022}, the projected langevin gradient descent algorithm purposed in \cite{Lamperski}
and our newly proposed Langevin Multiplicative Weights Update (LMWU) method.

Following this, we apply the estimated portfolio weights to the returns in the succeeding period and assess the performance of the constructed portfolio using the loss function defined in Equation \eqref{lossfun}. It's crucial to note that our loss function relies on a predetermined parameter, $\vect{\lambda}$, which represents different risk preferences. We take into account the following potential values for $\vect{\lambda}$:
\begin{enumerate}
\item Increasing preference: $\frac{1}{15}, \frac{2}{15},...,\frac{5}{15}$
\item Degenerate preference: $\frac{5}{15}, \frac{4}{15},...,\frac{1}{15}$
\item Mean-Variance (MV) preference: $\frac{1}{2},\frac{1}{2},0,0,0$
\item Mean-Variance-Skewness (MVS) preference: $\frac{1}{3},\frac{1}{3},\frac{1}{3},0,0$
\item Mean-Variance-Skewness-Kurtosis (MVSK) preference: $\frac{1}{4},\frac{1}{4},\frac{1}{4},\frac{1}{4},0$
\item Equal preference: $\frac{1}{5},\frac{1}{5},...,\frac{1}{5}$
\end{enumerate}

For each period $t$, we record the loss score, denoted as $\widehat{\textrm{Loss}}_{t}$, and compute the average loss score using the formula:
$
\widehat{\textrm{Score}} = \frac{1}{T-L}\sum^T_{t = T-L+1}\widehat{\textrm{Loss}}_{t}.
$
The outcomes are summarized in Table \ref{result}. The table's first column delineates the methods employed in the exercise, while columns two through seven display the results of these methods under various risk preferences, as indicated in the header row.

The data clearly demonstrates that the LMWU method outperforms the MWU method and several of its other variants across all risk preferences. For example, under the Degenerate preference, the LMWU method registers a score of 70.8930, a better result (considering the goal is minimization) than the MWU's score of 74.7203. This superiority is consistent across other risk preferences as well. Specifically, for Mean-Variance (MV) and Mean-Variance-Skewness (MVS) preferences, which are likely more commonplace in portfolio management, the LMWU method achieves superior scores (0.8464 and 0.8314, respectively) compared to the MWU method (0.8561 and 0.8391, respectively). Similar better performance can also be observed with Langevin MWU compared to other variants of MWU algorithms. These observations are consistent with our theoretical analysis of LMWU: it has the ability to escape local minima and converge towards global minima.

In summary, these results underscore the efficacy of our proposed LMWU method in the realm of polynomial portfolio optimization.

\section{Additional Experiments}

In this section, we present experiments comparing Langevin-MWU  with algorithms presented in Table \ref{tab:comparison}. We use Langevin-MWU and other algorithms for comparison to optimize several non-convex functions with many local minima. The experimental results show Langevin-MWU escapes such bad local minima and finds minima with smaller function values, while other algorithms either get stuck at local minima or are more unstable than Langevin-MWU. The experimental results are presented in Figure \ref{E3} and  Figure \ref{E4}. Future experiments, especially the examples demonstrating how the trajectories of Langevin-MWU avoid local minima and converge to global minima, are presented in the Appendix.

\paragraph{Test functions.}
We construct non-convex functions to verify the efficiency of LMWU in finding global minima. The functions are given as follows:
\begin{align*}
f_1 (x,y,z) &= -\ln \left( e^{ -10 (x - 0.3)^2 - 20(y - 0.5)^2 - 30(z - 0.2)^2 } \right. \\
    &\quad + e^{ -30 (x - 0.4)^2 - 20(y - 0.2)^2 - 36(z - 0.4)^2 } \bigg) \\
    &\quad + y + 10.
\end{align*}
and
\begin{align*}
f_2 (x,y,z) &= -(x - 0.6)^2 (x - 0.2)^2 + (y - 0.3)(y - 0.4)^3 \\
    &\quad + (z - 0.2)^3(z - 0.8) \\
    &\quad - xy - 0.4z.
\end{align*}

\begin{figure}[htbp]
    \centering
    \begin{subfigure}[b]{0.49\linewidth}
        \centering
        \includegraphics[width=\linewidth]{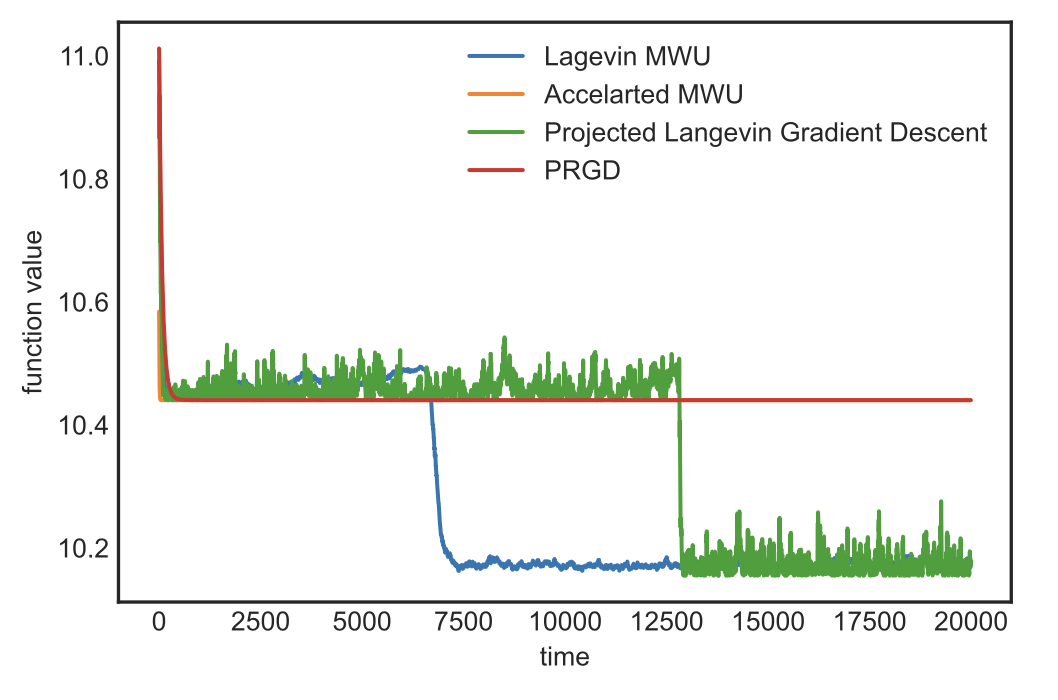}
        \caption{Test function : $f_1$}
        \label{E11}
    \end{subfigure}
    \hfill
    \begin{subfigure}[b]{0.49\linewidth}
        \centering
        \includegraphics[width=\linewidth]{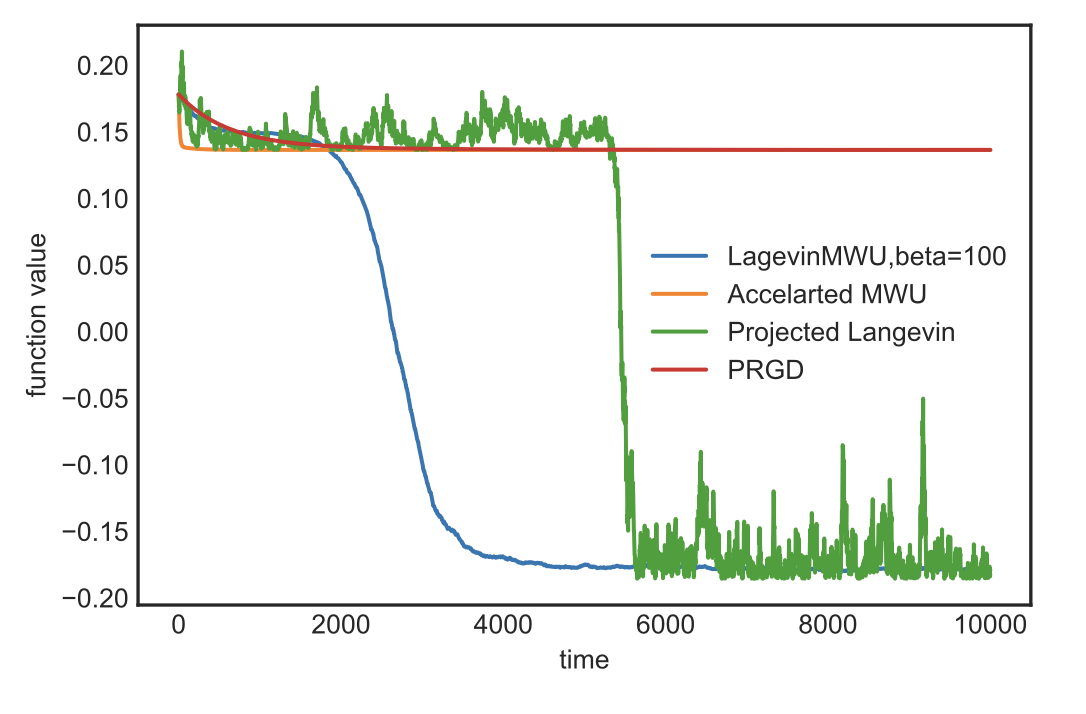}
        \caption{Test function : $f_2$}
        \label{E21}
    \end{subfigure}
    \caption{Comparison of LMWU with Accelerated MWU  \cite{FPW2022}, Projected Lagevin \cite{Lamperski}, and PRGD \cite{CB19}.}
    \label{E3}
\end{figure}

\begin{figure}[htbp]
    \centering
    \begin{subfigure}[b]{0.49\linewidth}
        \centering
        \includegraphics[width=\linewidth]{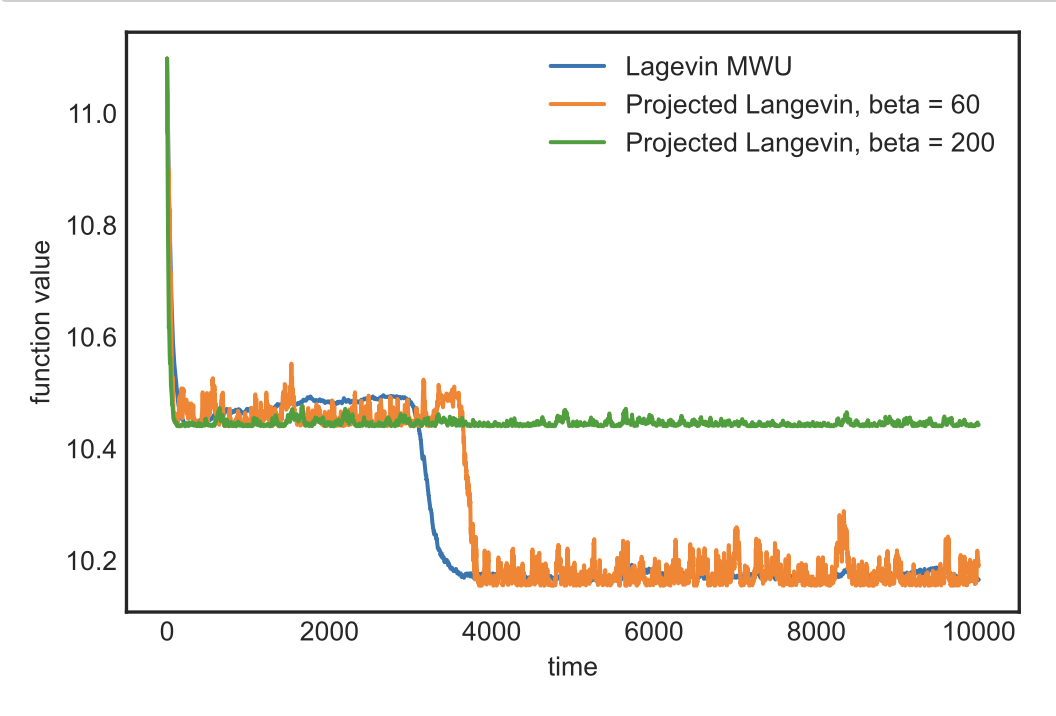}
        \caption{Test function : $f_1$}
        \label{E31}
    \end{subfigure}
    \hfill
    \begin{subfigure}[b]{0.49\linewidth}
        \centering
        \includegraphics[width=\linewidth]{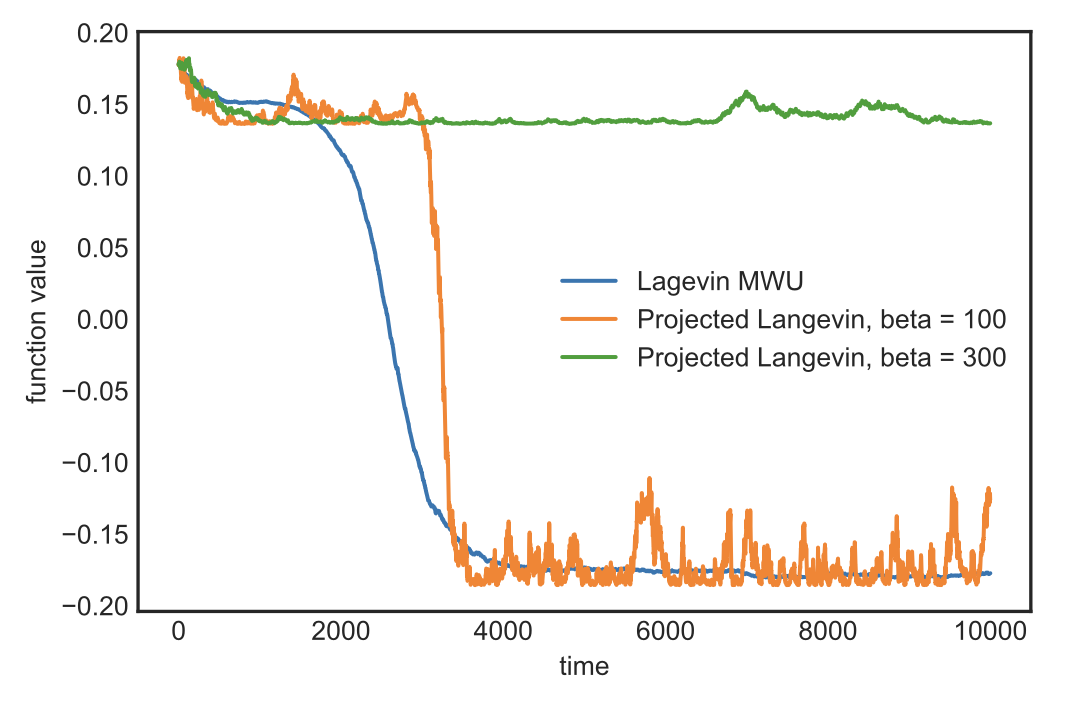}
        \caption{Test function : $f_2$}
        \label{E32}
    \end{subfigure}
    \caption{Future comparison of LMWU with Projected Langevin \cite{Lamperski}.}
    \label{E4}
\end{figure}

As shown in Figure \ref{E3} and \ref{E4}, LMWU and Projected Langevin can converge to global optima, but Perturbed RGD and Accelerated MWU only converge to local optima, which agree with the claims of \cite{FPW2022} and \cite{CB19}.


\section{Conclusion} In this paper we focus on a constrained non-convex optimization problem that widely exists in multi-agent learning. We propose a novel algorithm called Langevin Multiplicative Weights Update (LMWU) which is a stochastic version of classic MWU algorithm. Our theoretical analysis shows that LMWU converges to interior global optima of the objective function. Another important setting that is missing in current work is the time-varying environment, e.g., \cite{FFHLPW} in min-max optimization. We leave the time-varying portfolio management for future investigation. 
\section*{Acknowledgements}Xiao Wang acknowledges Grant 202110458 from Shanghai University of Finance and Economics and support from the Shanghai Research Center for Data Science and Decision Technology.
Xie's research is supported by the Natural Science Foundation of China (72173075) and the Shanghai Research Center for Data Science and Decision Technology.

\bibliography{aaai25}

\appendix
\onecolumn

\section{Multi-Agent L-MWU}
It is immediate to have the multi-agent version of Langevin MWU based on the algorithm given in main article.
\begin{algorithm}
\caption{Langevin-MWU (multi-agent) }
\label{alg:C2}
\begin{algorithmic}
\STATE Input: error bound $\delta>0$, $\beta>0$, step size $\epsilon>0$ 
\\
Initialize: $\vec{x}_i(0)\sim\rho_i(0)$ for all $i\in[N]$
\REPEAT
\STATE For $i=1,...,N$, 
\\
Compute 
$S_{\vec{x}_i}=\sum_{s=1}^{n_i}\frac{1}{x_{is}}$, 
\\
Sample $z_0^{is}\sim\mathcal{N}(0,1)$.
\\
Compute \[V_0^{is}=\frac{\epsilon}{2\beta}\left(n_i+1-(1+x_{is})S_{\vec{x}_i}\right)+\sqrt{2\epsilon\beta^{-1} x_{is}}z_0^{is}\]
\\
Set \[x_{is}\leftarrow\frac{x_{is}-\epsilon x_{is}\frac{\partial f}{\partial x_{is}}+V_0^{is}}{1-\epsilon\sum_{s=1}^{n_i}x_{is}\frac{\partial f}{\partial x_{is}}+\sum_{s=1}^{n_i}V_0^{is}}\]
\UNTIL{$k$ large enough}
\end{algorithmic}
\end{algorithm}

\section{Empirical Illustration on Polynomial Portfolio Management}
Supplementing to experiment of LMWU on polynomial portfolio management, we conduct empirical illustration of the non-convexity problem in polynomial portfolio optimization. We consider three representative stocks ($N=3$): (i) AAPL: Apple Inc.; (ii) AMT: American Tower Corp.; and (iii) COST: Costco Wholesale Corp., which are typical companies from the IT sector, the real estate sector, and the consumer discretionary sector, respectively. We collect the daily return data for the three stocks in year 2012 as an example. We consider a high order polynomial function of $d = 5$ in this illustration. We can even broaden out constraints on the weights by setting $w^L = -1$, $w^U = 2$, and $\lambda_i = 1/d$ for all $i$. 

We first consider a simple bivariate portfolio \{AAPL, AMT\}. We consider all the possible weights combination in the feasible set $\mathcal{W}$. Figure \ref{f1} plots the estimated $ -\hat{\mathbb{E}}[f(\vec{ w}, \vec{r})]$ against $w_1$.\footnote{Note that once $w_1$ is known, we immediately know the weight for the second stock due to $\sum^N_{i=1} w_i = 1$. } It is obvious that the figure consists of one global maximum along with a local maximum. The optimization problem to find the best weights is clearly non-convex. 

We include all three stocks \{AAPL, AMT, COST\} in the second portfolio. Similarly, we consider all the possible weights combination in the feasible set $\mathcal{W}$. The 3D Figure \ref{f2} plots the estimated $ -\hat{\mathbb{E}}[f(\vec{ w}, \vec{ r})]$ against $w_1$ and $w_2$. To have better visualization, we expand the interval of $w_1$ and $w_2$, although our estimation straightly follow the condition $\vec{ w}\in\mathcal{W}$. Again, the polynomial portfolio optimization problem is clearly non-convex. In fact, three maximums (one global and one local) appear in the graph. Although impossible to illustrate, we should expect more local maximums in such non-convexity as we include more stocks in the portfolio. 
Conventional convex optimization algorithm will clearly fail to deliver the global optima in the above illustration. This emphasizes the importance of our proposed algorithm that is capable of solving the non-convexity problem in high-order polynomial portfolio optimization.

\begin{figure}[H]
\centering
\includegraphics[scale = 0.6]{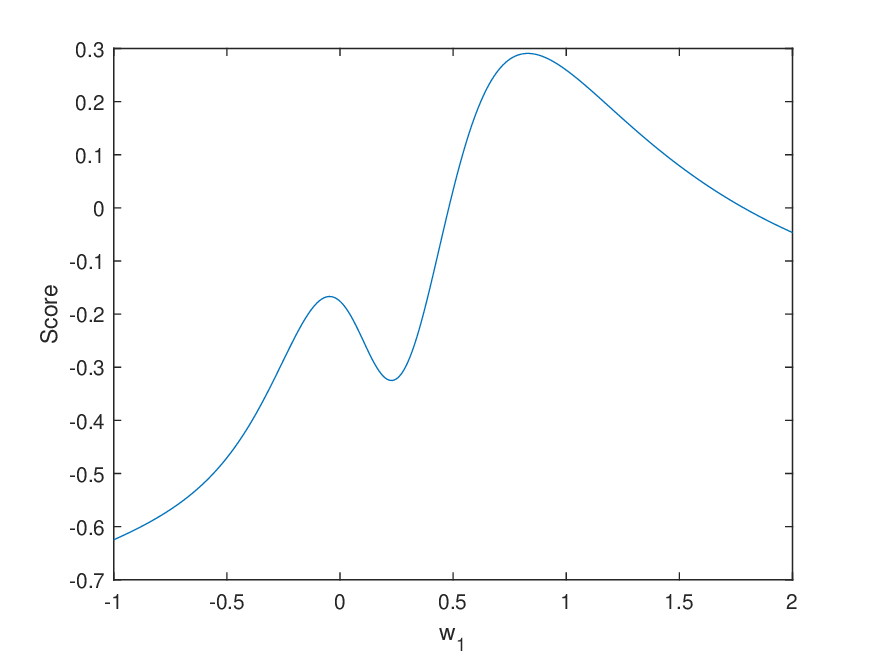}
\caption{An Illustration of Non-convexity in Polynomial Portfolio Optimization}\label{f1}
\end{figure}

\begin{figure}[H]
\centering
\includegraphics[scale = 0.6]{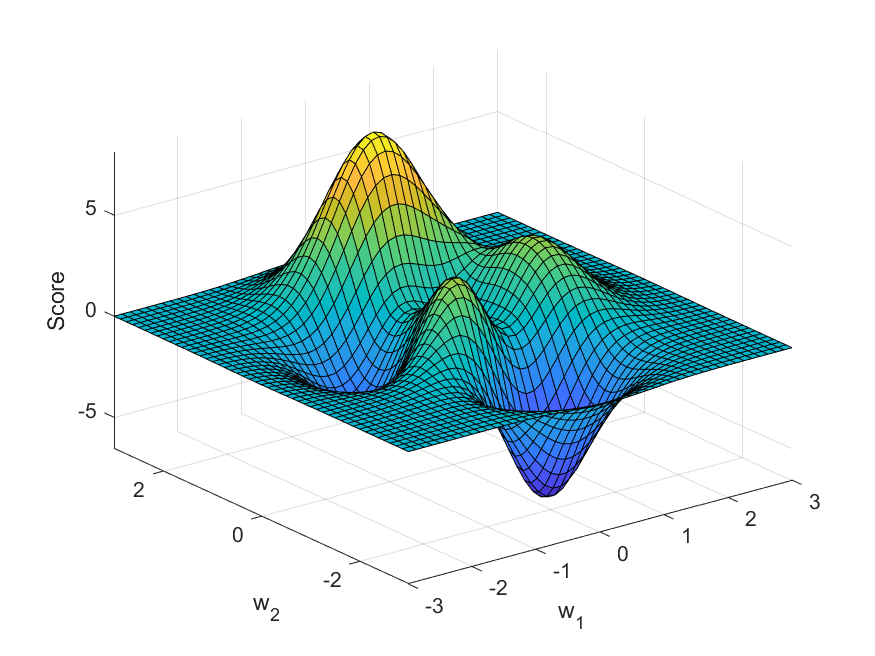}
\caption{A 3D Illustration of Non-convexity in Polynomial Portfolio Optimization}\label{f2}
\end{figure}

\section{More Experiments}

In this section, we present experiments comparing Langevin-MWU  with classic MWU. We use Langevin-MWU and MWU to optimize several non-convex functions with many local minima. The experimental results show that MWU converges to bad local minima, but Langevin-MWU escapes such bad local minima and finds minima with smaller function values.

\paragraph{Test functions.}
We construct 6 non-convex functions to verify the efficiency of LMWU in finding global minima. The functions are given as follows,
\begin{itemize}
\item[(1)]In Figure \eqref{f_1}, 
\begin{align}
    f_1(x,y,z) &= -\ln (e^{ -10 (x - 0.3)^2 - 20(y - 0.5)^2 - 30(z - 0.2)^2 }\notag \\
    &+e^{ -30 (x - 0.4)^2 - 20(y - 0.2)^2 - 36(z - 0.4)^2 } )\notag\\ &+ y + 10.
\end{align}
\item[(2)]In Figure \eqref{f_2} 
\begin{align}
f_2(x,y,z) &= -\ln (e^{ -15 (x - 0.4)^2 - 60(y - 0.4)^2 - 10(z - 0.2)^2 }\notag\\
& +e^{ -3 (x - 0.4)^2 - 2(y - 0.2)^2 - 6(z - 0.4)^2 } ) + y.
\end{align}

\item[(3)]In Figure \eqref{f_3} 
\begin{align}
f_3(x,y,z) &= (x - 0.3)^2 (x - 0.9)^2 \notag\\ 
&+ (y - 0.2)^2(y - 0.7)^2\notag\\ 
&+ (z - 0.6)^2(z - 0.1)^2 + (x - 0.3)(y - 0.5)   .
\end{align}

\item[(4)] in Figure \eqref{f_4} 
\begin{align}
f_4(x,y,z) &= -(x - 0.6)^2 (x - 0.2)^2 \notag\\
&+ (y - 0.3)(y - 0.4)^3\notag\\ 
&+ (z - 0.2)^3(z - 0.8) - xy - 0.4z  .
\end{align}

\item[(5)]In (a) of Figure \eqref{f_56} 
\begin{align}
f_5(x,y,z,w,v) &= (x - 0.6)^2 (x - 0.2)^2 - xy  \notag\\
&+ (y - 0.3)^2(y - 0.4)^2 + (z - 0.2)^4
\notag\\
&- 0.5zw + (w - 0.5)^4 + (v - 0.3)^4   .
\end{align}

\item[(6)]In (b) of Figure \eqref{f_56} 
\begin{align}
f_6(x,y,z,w,v,h) &= (x - 0.6)^2 (x - 0.8)  \notag\\
&+ (y - 0.9)(y - 0.4)^2 \notag\\
&+ (z - 0.2)^2 + (v - 0.6)^2 
\notag\\
&+ (w - 0.5)^2 - 0.5vw + (h - 0.5)^2  .
\end{align}
\end{itemize}

For test functions (1)-(4), we present both trajectories of algorithms on the contour maps of text functions and the curve of convergence in the function values. With simplex constrains $x + y + z = 1, x,y,z \ge 0$, the values of a three variables function $f(x,y,z) $ constrained on a simplex are determined variables $x$ and $y$, thus we can draw trajectory and level curves of $f(x,y,z)$ on a $(x,y)$-plane. Note that since $x + y \le 1$, only the lower half part of the $(x,y)$-plane is meaningful, and algorithms' trajectories will only appear on lower half part of $(x,y)$-plane. For test functions (5) and (6) with more than three variables, we only show their curves of convergence in function values.
\paragraph{Parameter setting.} 
Since the behaviors of Langevin-MWU are controlled by the parameter $\beta$, in the experiments we choose different $\beta$ to show the power of using larger $\beta$'s, the choices of $\beta$ are denoted on the convergence curves graph. In experiments we set the parameters as follows:
\begin{itemize}
\item[(1)] Figure \eqref{f_1}: Initial point $(0.3,0.6,0.1)$, MWU's step size $10^{-3}$, LMWU's step size $10^{-4}$ , $\beta  =10,50,100 $.
\item[(2)] Figure \eqref{f_2}: Initial point $(0.4,0.1,0.5)$, MWU's step size $10^{-3}$, LMWU's step size $5 \times 10^{-5}$ , $\beta  =10,50,100 $.
\item[(3)] Figure \eqref{f_3}: Initial point $(0.2,0.75,0.05)$, MWU's step size $10^{-2}$, LMWU's step size $10^{-3}$ , $\beta  =10,2000,5000 $.
\item[(4)] Figure \eqref{f_4}: Initial point $(0.5,0.4,0.1)$, MWU's step size $10^{-2}$, LMWU's step size $2 \times 10^{-4}$ , $\beta  =1000,2000,8000 $.
\item[(5)] (a) of Figure \eqref{f_56} :  Initial point $(0.1,0.05,0.4,0.4,0.05)$, MWU's step size $5 \times 10^{-2}$, LMWU's step size $5 \times 10^{-3}$ , $\beta  =800,2000,3000 $.
\item[(6)] (b) of Figure \eqref{f_56} :  Initial point $(0.4,0.1,0.1,0.2,0.1,0.1)$, MWU's step size $10^{-4}$, LMWU's step size $10^{-4}$ , $\beta=300,3000,8000 $.
\end{itemize}
In each experiments MWU and LMWU starting from the same initial points and run the same number of steps.

\paragraph{Analysis of experimental results.} The curves of convergence in function values show that Langevin-MWU outperforms MWU when $\beta$ is small enough. This can be seen clearly from the trajectories of the algorithms on contour map : MWU is  attracted by a local minimum near initial points, but although starting from the same initial point, LMWU will escape the bad local minimum near initial points and go to the global minimum. The choice of $\beta$ have a great influence on behaviors of LMWU : as the experimental results show that larger $\beta$ will make LMWU find a better convergence point, but with a slower convergence rate, this is compatible with our theoretical analysis. Moreover, as shown in Figure \eqref{f_3} and (a) of Figure \eqref{f_56}, an inappropriate choice of $\beta$ will make LMWU underperform MWU. In fact, for different test functions, the range of suitable $\beta$ is very different, thus choosing an optimal $\beta$ is an important question for future research.

\begin{figure}[H]
\centering
\begin{subfigure}[b]{0.3\textwidth}
    \centering
    \includegraphics[width=\textwidth]{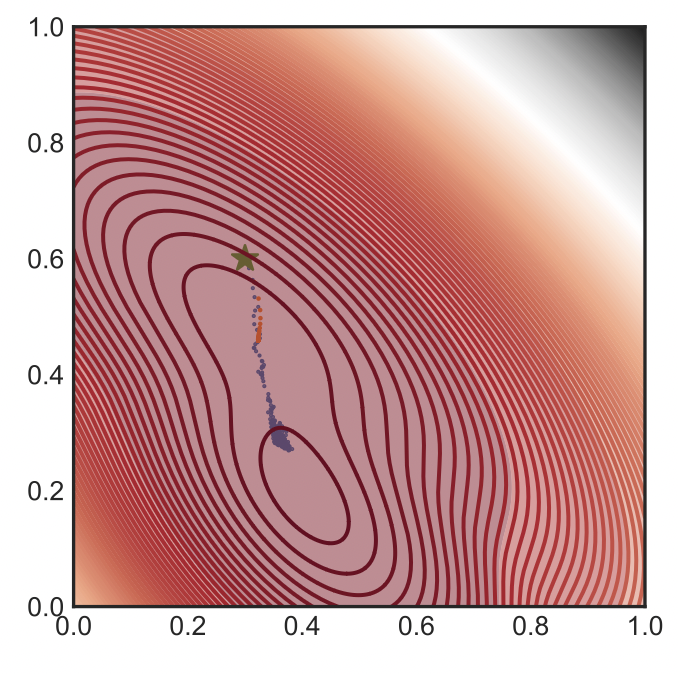}
    \caption{Trajectories for $f_1$}
    \label{f1_trajectories}
\end{subfigure}
\begin{subfigure}[b]{0.3\textwidth}
    \centering
    \includegraphics[width=\textwidth]{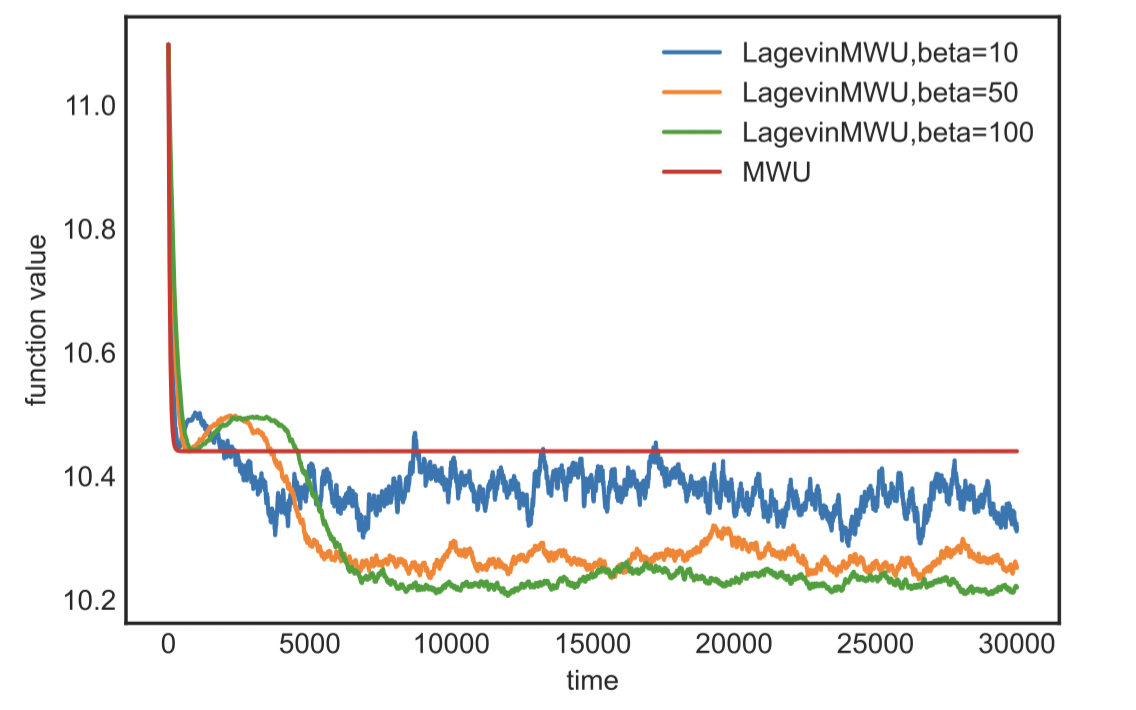}
    \caption{Convergence curves for $f_1$}
    \label{f1_convergence}
\end{subfigure}
\caption{Test function $f_1$}
\label{f_1}
\end{figure}

\begin{figure}[H]
\centering
\begin{subfigure}[b]{0.3\textwidth}
    \centering
    \includegraphics[width=\textwidth]{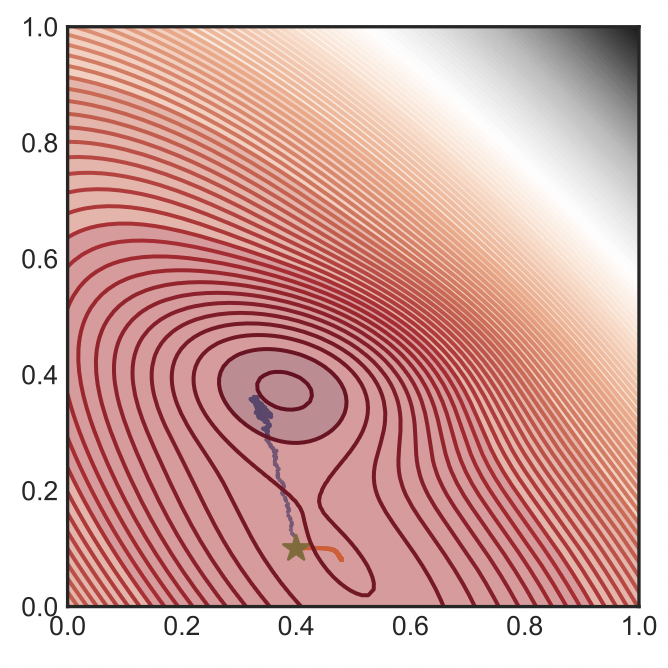}
    \caption{Trajectories for $f_2$}
    \label{f2_trajectories}
\end{subfigure}
\begin{subfigure}[b]{0.3\textwidth}
    \centering
    \includegraphics[width=\textwidth]{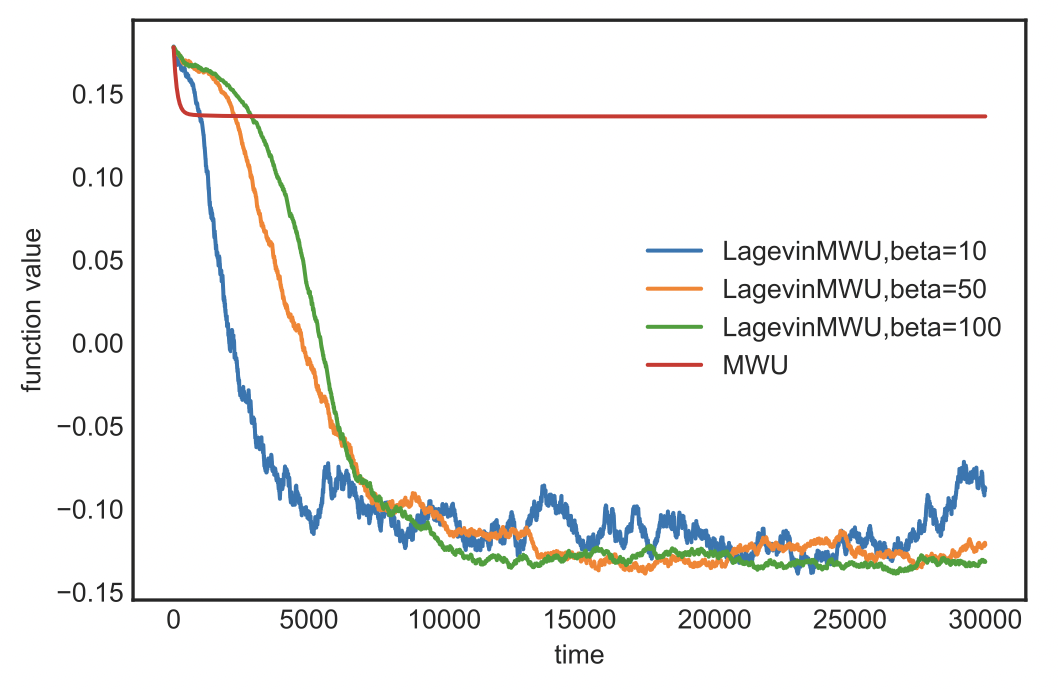}
    \caption{Convergence curves for $f_2$}
    \label{f2_convergence}
\end{subfigure}
\caption{Test function $f_2$}
\label{f_2}
\end{figure}

\begin{figure}[H]
\centering
\begin{subfigure}[b]{0.3\textwidth}
    \centering
    \includegraphics[width=\textwidth]{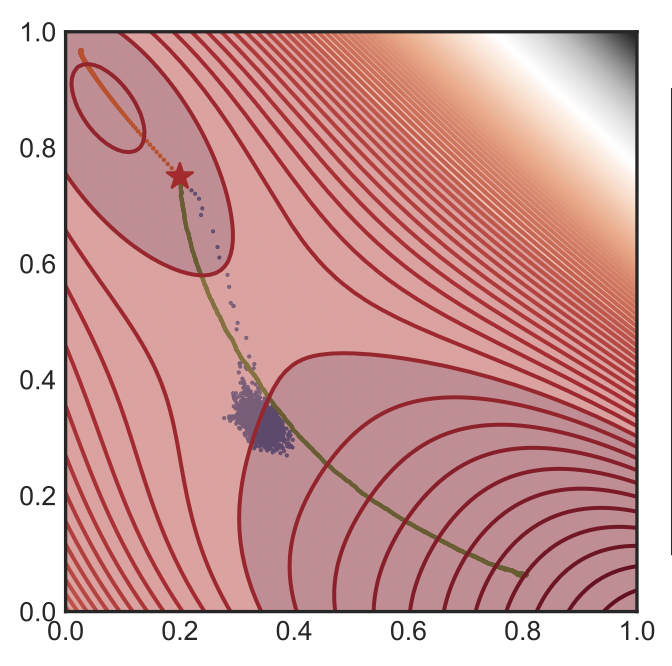}
    \caption{Trajectories for $f_3$}
    \label{f3_trajectories}
\end{subfigure}
\begin{subfigure}[b]{0.3\textwidth}
    \centering
    \includegraphics[width=\textwidth]{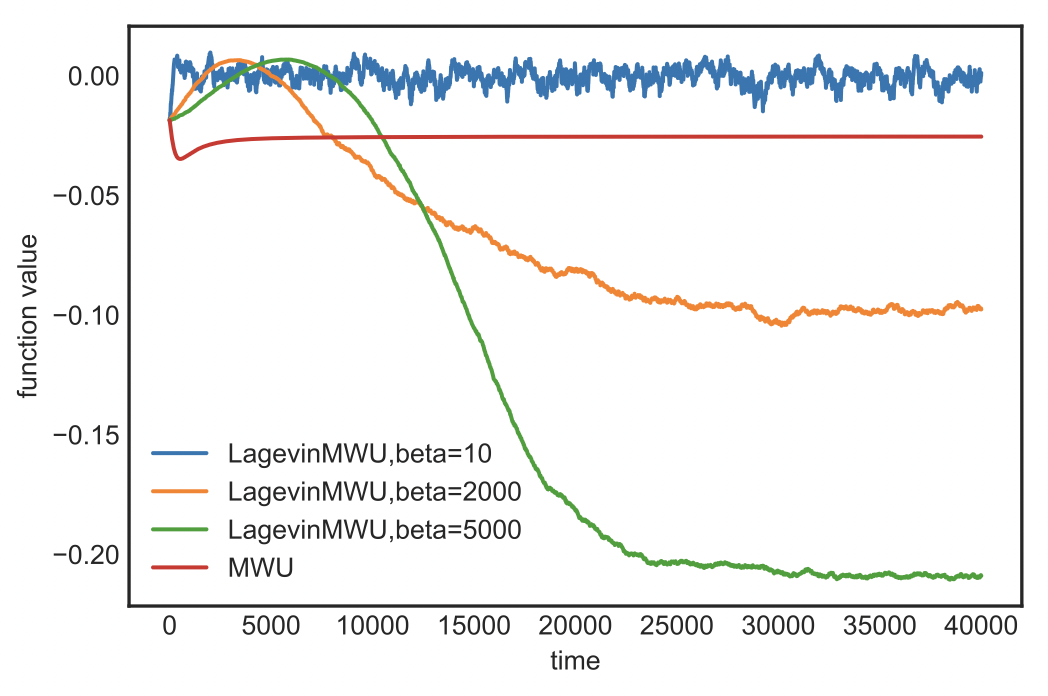}
    \caption{Convergence curves for $f_3$}
    \label{f3_convergence}
\end{subfigure}
\caption{Test function $f_3$}
\label{f_3}
\end{figure}

\begin{figure}[H]
\centering
\begin{subfigure}[b]{0.3\textwidth}
    \centering
    \includegraphics[width=\textwidth]{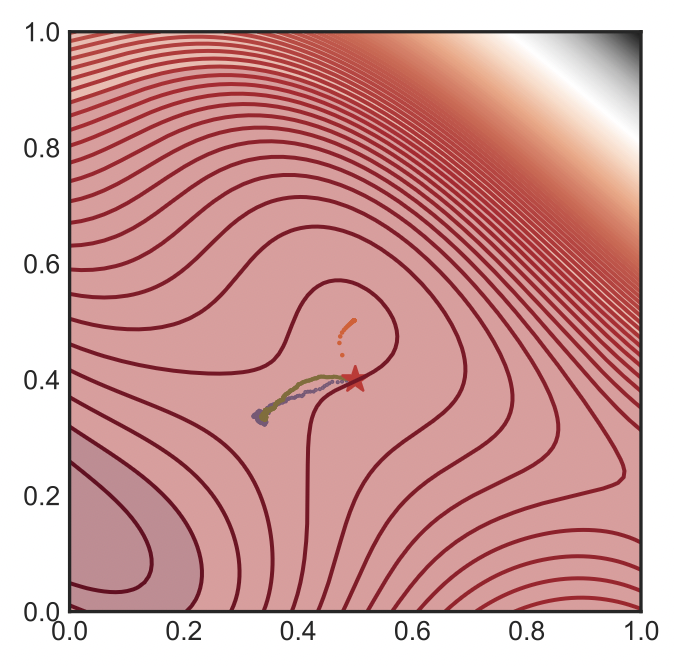}
    \caption{Trajectories for $f_4$}
    \label{f4_trajectories}
\end{subfigure}
\begin{subfigure}[b]{0.3\textwidth}
    \centering
    \includegraphics[width=\textwidth]{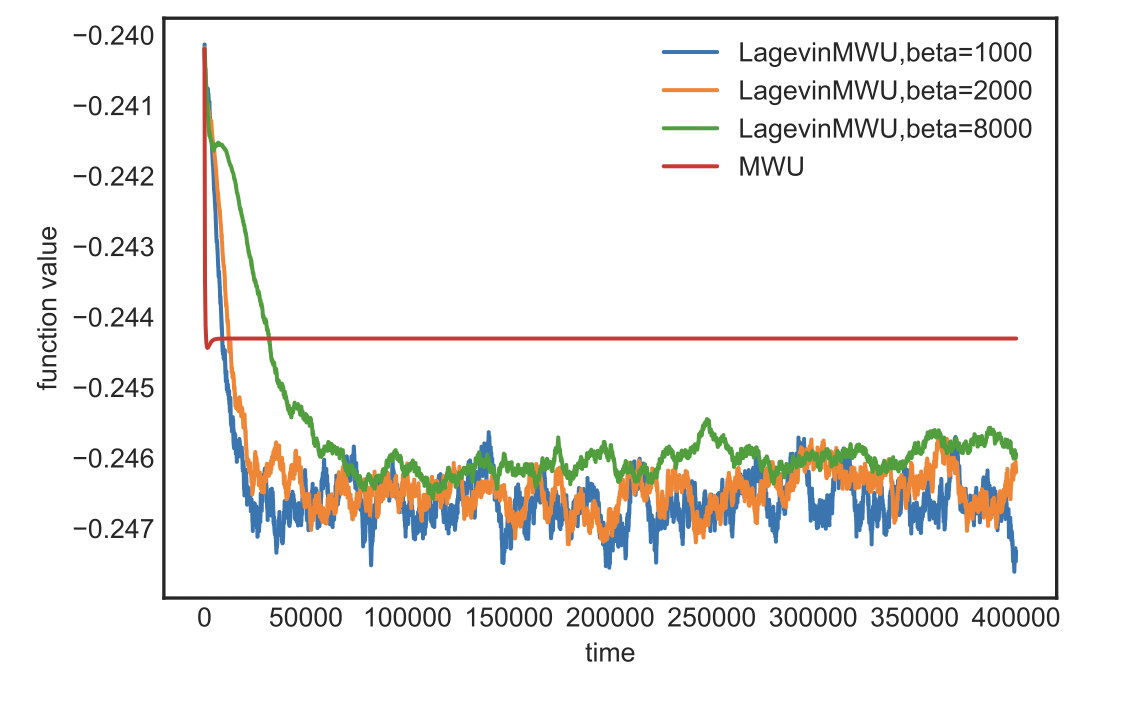}
    \caption{Convergence curves for $f_4$}
    \label{f4_convergence}
\end{subfigure}
\caption{Test function $f_4$}
\label{f_4}
\end{figure}

\begin{figure}[H]
\centering
\begin{subfigure}[b]{0.3\textwidth}
    \centering
    \includegraphics[width=\textwidth]{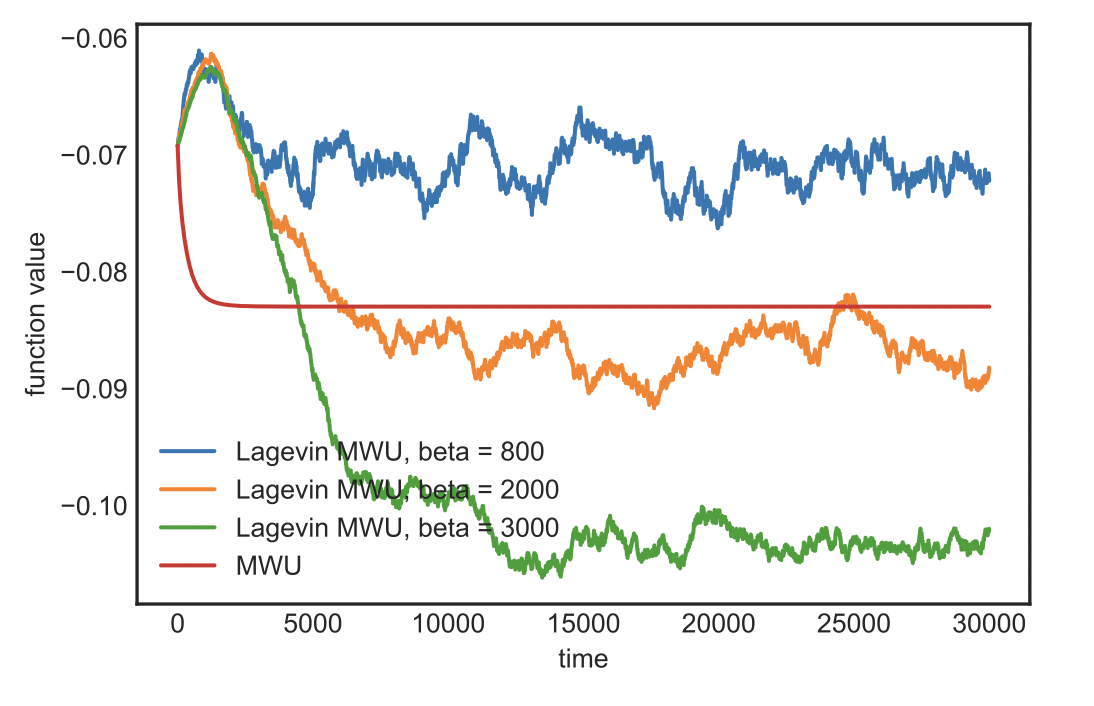}
    \caption{Convergence curves of $f_5$}
    \label{f5_convergence}
\end{subfigure}
\begin{subfigure}[b]{0.3\textwidth}
    \centering
    \includegraphics[width=\textwidth]{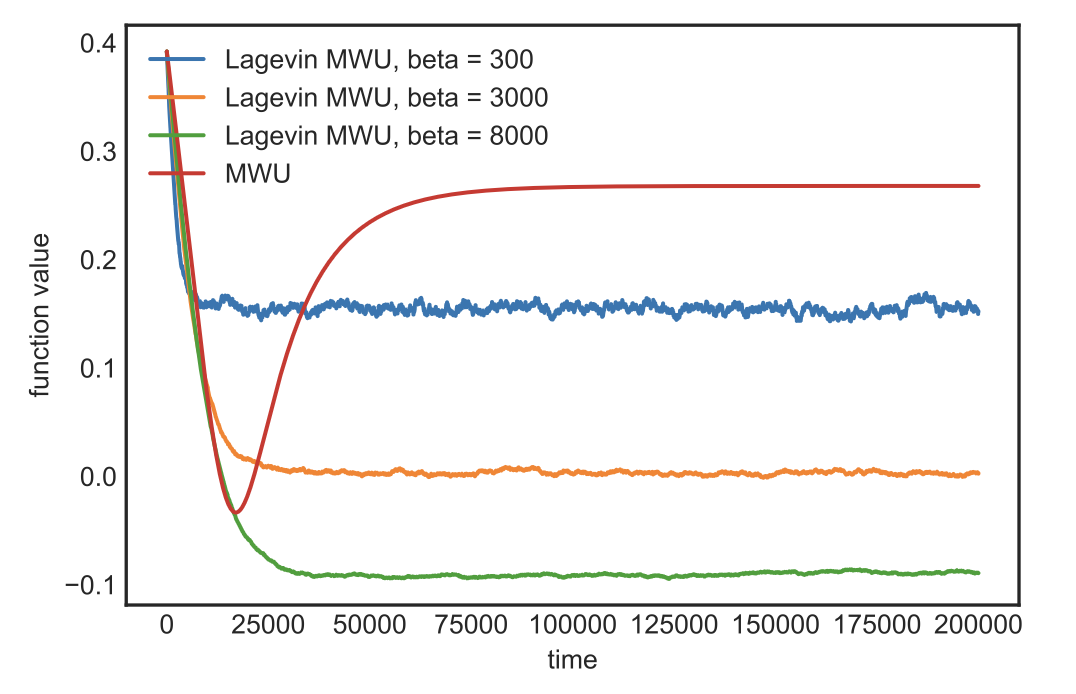}
    \caption{Convergence curves of $f_6$}
    \label{f6_convergence}
\end{subfigure}
\caption{Test functions $f_5$ and $f_6$}
\label{f_56}
\end{figure}

\section{More Background on Analysis on Manifolds}
The standard references for the necessary background are \cite{JLee,Hsu}.
\subsection{Covariant Derivative}
Given two vector fields $X$ and $Y$ on $\mathcal{M}$, the covariant derivative, $D_XY$ is a bilinear operator with the following properties:
\begin{align*}
D_{\alpha_1X_1+\alpha_2X_2}Y&=\alpha_1D_{X_1}Y+\alpha_2D_{X_2}Y,
\\
D_X(Y_1+Y_2)&=D_X(Y_1)+D_X(Y_2),
\\
D_X(\alpha Y_1)&=\alpha D_X(Y_1)+X(\alpha)Y_1.
\end{align*}
In local coordinate systems with variable $x$, the vector field $X$ can be represented as $X=\sum X_i\partial x_i$, where $\partial x_i$ are the basis vector fields, and $Y=\sum Y_i\partial x_i$, the covariant derivative is given by 
\[
D_XY=\sum_k\left(X(Y_k)+\sum_i\sum_jX_iY_j\Gamma^k_{ij}\right)\partial x_k.
\]
where the Christoffel symbols $\Gamma_{ij}^k$ will be used and computed explicitly in deriving the Langevin MWU algorithm. With the basis $\{\partial x_i\}$, the covariant derivative and Christoffel symbols are related as follows,
\[
D_{\partial x_j}\partial x_i=\sum_k\Gamma_{ij}^k\partial x_k
\]
and the $\Gamma_{ij}^k$ can be computed by
\[
\Gamma_{ij}^k=\frac{1}{2}\sum_mg^{km}(\partial_jg_{mi}+\partial_ig_{mj}-\partial_mg_{ij}).
\]

\subsection{Differential Operators on Manifold}
In local coordinate systems, denote $\abs{g}=\det{g_{ij}(x)}$, the differential operators on manifold can be written as

\[
\Div V=\frac{1}{\sqrt{\abs{g}}}\sum_i\frac{\partial}{\partial x_i}\left(\sqrt{\abs{g}}V_i\right)
\]

The Laplace-Beltrami operator is

\begin{align}
\Delta_{M}f&=\Div(\grad f)
\\
&=\frac{1}{\sqrt{\abs{g}}}\sum_i\frac{\partial}{\partial x_i}\left(\sqrt{\abs{g}}\sum_jg^{ij}\frac{\partial f}{\partial x_j}\right)
\\
&=\sum_{i}b_i\frac{\partial f}{\partial x_i}+\sum_{i,j}g^{ij}\frac{\partial^2f}{\partial x_i\partial x_j},
\end{align}
where
\[
b_i=\frac{1}{\sqrt{\abs{g}}}\sum_j\frac{\partial (\sqrt{\abs{g}}g^{ij})}{\partial x_j}=\sum_{j,k}g^{jk}\Gamma^i_{jk}
\]
\section{Derivation of Algorithm}
The standard geometry used in simplex is a special type of Riemannian geometry on the positive orthant and interior of simplex, called \emph{Shahshahani geometry} \cite{Shahshahani,HofSig}. The metric matrix $\{g_{ij}(\vec{x})\}$ on $\mathbb{R}_+^{d}=\{\vec{x}:x_i>0\text{ for all }i\in[d]\}$ is diagonal with $g_{ii}(\vec{x}=\frac{\abs{\vec{x}}}{x_i})$ where $\abs{\vec{x}}
=\sum_jx_j$. Use the Riemannian gradient \eqref{grad}, we have the explicit form of the Shahshahani gradient for $\vec{x}\in\mathbb{R}_+^d$ as follows:
\[
\grad f(\vec{x})=g^{-1}\nabla f(\vec{x})=\left(\frac{x_1}{\abs{\vec{x}}}\frac{\partial f}{\partial x_1},...,\frac{x_d}{\abs{\vec{x}}}\frac{\partial f}{\partial x_d}\right).
\]

Viewing $\Delta_+^{d-1}$ as a Riemannian submanifold of $\mathbb{R}_+^{d}$, we endow the simplex with a Riemannian metric whose matrix satisfies $g_{ii}(\vec{x})=\frac{1}{x_i}$ on diagonal and $g_{ij}=0$ on all other entries. The tangent space of $\Delta_+^{d-1}$ at $\vec{x}$ is denoted by $T_{\vec{x}}$ which consists of all the vectors $\vec{v}=(v_1,...,v_d)$ such that $\sum_jv_j=0$. Thus the tangent space $T_{\vec{x}}\Delta_+^{d-1}$ is identified with the hyperplane passing through $\vec{0}$ and parallel to $\Delta_+^{d-1}$. In the derivation of the L-MWU algorithm, the geometric property used most frequently is the othogonality in $\mathbb{R}_+^d$. Let $\langle\cdot,\cdot\rangle_{\vec{x}}$ be the Riemannian metric (a space-dependent inner product on $T_{\vec{x}}\Delta_+^{d-1}$), i holds that for all $\vec{u}\in T_{\vec{x}}\Delta_+^{d-1}$ and any $\lambda\ne 0$, we have $\langle\vec{u},\lambda\vec{x}\rangle_{\vec{x}}=0$. This means that the straight line passing through $\vec{0}$ and $\vec{x}$ is orthogonal to the tangent space of $\Delta_+^{d-1}$, with respect to the Shahshahani metric on $\mathbb{R}_+^d$. With these background in Shahshahani geometry of simplex, MWU can be viewed as the Riemannian gradient descent, especially the linear variant \eqref{MWUclassic} is the Riemannian gradient descent with retraction as the projection mapping from tangent space onto the base manifold.
\subsection{Brownian motion in Riemannian manifold}
In locally coordinate systems, Brownian motion can be written in the following way:
\[
dX_t=-\frac{1}{2}g^{ij}\Gamma^k_{ij}dt+\sqrt{g^{-1}}dB_t
\]

The standard Brownian motion in $\mathbb{R}^n$ can be generated by the diffusion equation 
\[
\frac{\partial \rho}{\partial t}=\frac{1}{2}\Delta \rho.
\]

The Brownian motion on manifold can be seen to be generated by the Laplace-Beltrami operator. Let $\sigma=\{\sigma_{ij}\}$ be the unique symmetric square root of $g^{-1}=\{g^{ij}\}$. In local coordinate, the solution of the stochastic differential equation for a process $x_t=(x_t^1,...,x_t^n)$:
\[
dx_t^i=\frac{1}{2}b_i(x_t)dt+\sum_j\sigma_{ji}(x_t)dB^j_t,
\]
that is a diffusion process generated by $\frac{1}{2}\Delta_M$, i.e. $x_t$ is a Brownian motion on $M$.

\begin{theorem}[Fokker-Planck Equation]
For any stochastic differential equation of the form
\[
dx_t=\mu(x_t,t)dt+\sqrt{A(x_t,t)}dB_t,
\]
the probability density of the SDE is given by the diffusion equation
\[
\frac{\partial\rho(x,t)}{\partial t}=-\sum_{i=1}^n\frac{\partial}{\partial x_i}\left(\mu_i(x,t)\rho(x,t)\right)+\frac{1}{2}\sum_{i=1}^n\sum_{j=1}^n\frac{\partial^2}{\partial x_i\partial x_j}\left(A_{ij}(x,t)\rho(x,t)\right)
\]
\end{theorem}

\subsection{Langevin algorithm in Shahshahani manifold}
In this subsection we show that the algorithm \ref{alg:C} is discretization and approximation of Langevin based Riemannian gradient descent on Shahshahani manifold.
\begin{proposition}
Algorithm \ref{alg:C} that is given as follows,
\[
x_i\leftarrow\frac{x_i-\epsilon x_i\frac{\partial f}{\partial x_i}+V_0^i}{1-\epsilon\sum_{j=1}^nx_j\frac{\partial f}{\partial x_j}+\sum_{j=1}^nV_0^j}
\]
is Langevin gradient descent that scaled with Shahshahani geometry in simplex.
\end{proposition}
\begin{proof}
MWU with noise can be implemented by gradient descent on the simplex with Shahshahani metric with a Gaussian noise in the tangent space. Consider the open simplex $\Delta_n\subset\mathbb{R}^{n}$,

\[
\Delta_n=\left\{\vec{x}=(x_1,...,x_n)\in\mathbb{R}^n:x_i>0\ \ \text{for all}\ \ i,\sum_{i=1}^nx_i=1\right\}
\]
and its tangent spaces are identified with the hyperplane through origin,
\[
T_n=\left\{\vec{v}=(v_1,...,v_n)\in\mathbb{R}^n:\sum_{i=1}^n=0\right\}.
\]
The exponential map $\Exp:\Delta_n\times T_n\rightarrow\Delta_n$,
\[
(\vec{x},\vec{v})\mapsto\sum_{i=1}^n\frac{x_ie^{v_i}}{\sum_{j=1}^nx_je^{v_j}}\vec{e}_i.
\]
We will use the restrictions notation $\Exp_{\vec{x}}(\vec{v})$ as well. The coordinate system is given by 
\[
\varphi(\vec{x})=(x_1,...,x_{n-1}) 
\]
and the parametrization $\varphi^{-1}$ is
\[
\varphi^{-1}(x_1,...,x_{n-1})=(x_1,...,x_{n-1},1-\sum_{i=1}^{n-1}x_i).
\]
The metric matrix pulled back from $\Delta_n$ to the hypersurface $(x-1,...,x_{n-1},0)$ is given by
\[
\bar{g}_{ij}=(d\varphi^{-1})^{\top}g_{ij}d\varphi^{-1}, \ \ i,j\in [n-1].
\]
Next we derive the expression of noise based on the Brownian motion in the projected simplex $\varphi(\Delta_n)$. The Brownian motion (generated by full Laplace-Beltrami instead of $\frac{1}{2}\Delta_M$) pull back into the tangent space, according to the local expression, is
\[
dX_t=b(X_t)dt+\sqrt{2\bar{g}^{-1}}dB_t
\]
where the basis is $\vec{e}_1,...,\vec{e}_{n-1}$ and 
\[
b_i(X_t)=\sum_{j,k}\bar{g}^{jk}\Gamma^i_{jk}.
\]
Denote $\sigma_{ij}$ to be the unique matrix such that $\sigma^2=g^{-1}$, the Riemannian noise with stepsize $\epsilon$ can be written as
\[
W_0^i=\epsilon\sum_{j,k}\bar{g}^{jk}\Gamma^i_{jk}+\sqrt{2\epsilon}\sum_j\sigma_{ji}z^j_0
\]
where $z^j_0\sim\mathcal{N}(0,1)$. Then the push-forward noise of $W_0$ to $T_{\vec{x}}\Delta_n$ by $\varphi^{-1}$ is 
\[
\xi_0=d\varphi^{-1}(W_0).
\] 
One can compute the Christoffel symbols according to the rule
\[
\Gamma^i_{jk}=\frac{1}{2}\sum_m\left(\frac{\partial g_{km}}{\partial x_j}+\frac{\partial g_{mj}}{\partial x_k}-\frac{\partial g_{jk}}{\partial x_m}\right)g^{mi}.
\]
An simple approach to compute the noise is to think of the noise or Brownian motion on $\Delta_n$ as the projection of noise or Brownian motion in the ambient space $\mathbb{R}^n_+$ onto $\Delta_n$. Let $g_{ij}$ be the Shahshahani metric on $\mathbb{R}^n_+$, according to the previous derivation, denote $V_0$ as the noise, we have the following expression 
\[
V^i_0=\epsilon\sum_{j,k}g^{jk}\Gamma^i_{jk}+\sqrt{2\epsilon}\sum_j\sigma_{ji}z_0^j
\]
where $j\in[n]$ and $\Gamma^i_{jk}$ is the Christoffel symbol of $g_{ij}$. Next we compute them explicitly. Recall that $g_{ij}=\diag\{\frac{1}{x_i}\}$, and we have the following calculation:

One can compute the Christoffel symbols according to the rule
 \[
 \Gamma^i_{jk}=\frac{1}{2}\sum_m\left(\frac{\partial g_{km}}{\partial x_j}+\frac{\partial g_{mj}}{\partial x_k}-\frac{\partial g_{jk}}{\partial x_m}\right)g^{mi}
 \]
 Note that $g_{ij}=0$ for all $i\ne j$, then we have
 \begin{align}
 \Gamma^i_{jk}&=\frac{1}{2}\sum_m\left(\frac{\partial g_{km}}{\partial x_j}+\frac{\partial g_{mj}}{\partial x_k}-\frac{\partial g_{jk}}{\partial x_m}\right)g^{mi}
 \\
 &=\frac{1}{2}\left(\left(\sum_m\frac{\partial g_{km}}{\partial x_j}\right)g^{mi}+\left(\sum_m\frac{\partial g_{mi}}{\partial x_k}\right)g^{mi}-\left(\sum_m\frac{\partial g_{jk}}{\partial x_m}\right)g^{mi}\right)
 \\
 &=\frac{1}{2}\left(\frac{\partial g_{kk}}{\partial x_j}g^{kj}+\frac{\partial g_{jj}}{\partial x_k}g^{ji}-\frac{\partial g_{jk}}{\partial x_i}g^{ii}\right)
 \end{align} 
 
If $j=k\ne i$
\begin{align}
\Gamma_{jk}^i&=\Gamma_{jj}^i=\frac{1}{2}\left(\frac{\partial g_{jj}}{\partial x_j}g^{jj}+\frac{\partial g_{jj}}{\partial x_j}g^{ji}-\frac{\partial g_{jj}}{\partial x_i}g^{ii}\right)
\\
&=\frac{1}{2}\left(\frac{\partial g_{jj}}{\partial x_j}g^{jj}-\frac{\partial g_{jj}}{\partial x_i}g^{ii}\right)
\\
&=\frac{1}{2}\left(\frac{x_j-\abs{\vec{x}}}{x_j\abs{\vec{x}}}-\frac{x_i}{x_j\abs{\vec{x}}}\right)
\\
&=\frac{1}{2}\frac{x_j-x_i-\abs{\vec{x}}}{x_j\abs{\vec{x}}}
\end{align}

if $i=j\ne k$
\begin{align}
\Gamma_{jk}^i&=\Gamma_{ik}^i=\frac{1}{2}\left(\frac{\partial g_{ii}}{\partial x_k}g^{ii}-\frac{g_{ii}}{\partial x_i}g^{ii}\right)
\\
&=\frac{1}{2}\left(\frac{1}{\abs{\vec{x}}}-\frac{x_i-\abs{\vec{x}}}{x_i\abs{\vec{x}}}\right)
\\
&=\frac{1}{2}\left(\frac{x_i-x_j+\abs{\vec{x}}}{x_i\abs{\vec{x}}}\right)
\\
&=\frac{1}{2x_i}
\end{align}

if $i=k\ne j$

\begin{align}
\Gamma_{jk}^i&=\Gamma_{ji}^i=\frac{1}{2}\left(\frac{\partial g_{jj}}{\partial x_i}g^{ji}-\frac{\partial g_{ji}}{\partial x_i}g^{ii}\right)
\\
&=\frac{1}{2}\left(-\frac{\partial g_{ji}}{\partial x_i}g^{ii}\right)
\\
&=0
\end{align}

If $i\ne j\ne k$

\begin{align}
\Gamma_{jk}^i=0
\end{align}

If $i=j=k$

\begin{align}
\Gamma_{jk}^i=\frac{1}{2}\frac{x_i-\abs{\vec{x}}}{x_i\abs{\vec{x}}}.
\end{align}

Next we compute the terms in $V_0^i$. Since $g^{jk}=0$ if $j\ne k$, we have
\begin{align}
\sum_{j,k}g^{jk}\Gamma_{jk}^i&=\Gamma_{ii}^i+\sum_{j\ne i}g^{jj}\Gamma_{jj}^i
\\
&=\frac{1}{2}\frac{x_i-\abs{\vec{x}}}{x_i\abs{\vec{x}}}+\sum_{j\ne i}\frac{1}{2}\frac{x_j-x_i-\abs{\vec{x}}}{x_j\abs{\vec{x}}}
\\
&=\frac{1}{2\abs{\vec{x}}}\left(\frac{x_i-\abs{\vec{x}}}{x_i}+\sum_{j\ne i}\frac{x_j-x_i-\abs{\vec{x}}}{x_j}\right)
\\
&=\frac{1}{2\abs{\vec{x}}}\left(\frac{x_i-\abs{\vec{x}}}{x_i}+\sum_{j\ne i}\left(\frac{x_j-\abs{\vec{x}}}{x_j}-\frac{x_i}{x_j}\right)\right)
\end{align}
Since $\vec{x}$ is on the simplex, i.e., $\abs{\vec{x}}=1$, the above expression can be simplified to
\begin{align}
&\frac{1}{2}\left(\frac{x_i-1}{x_i}+\sum_{j\ne i}\left(1-\frac{1}{x_j}-\frac{x_i}{x_j}\right)\right)
\\
&=\frac{1}{2}\left(1-\frac{1}{x_i}+(n-1)-\sum_{j\ne i}\frac{1}{x_j}-\sum_{j\ne i}\frac{x_i}{x_j}\right)
\\
&=\frac{1}{2}\left(n-\frac{1}{x_i}-\sum_{j\ne i}\frac{1}{x_j}-\sum_{j\ne i}\frac{x_i}{x_j}\right)
\\
&=\frac{1}{2}\left(n-\sum_{j}\frac{1}{x_j}-\sum_{j\ne i}\frac{x_i}{x_j}-\frac{x_i}{x_i}+1\right)
\\
&=\frac{1}{2}\left(n+1-\sum_{j}\frac{1}{x_j}-\sum_{j}\frac{x_i}{x_j}\right),
\end{align} 
on the other hand, 
\[
\sigma_{ii}=\sqrt{x_i}\ \ \text{and}\ \ 0\ \text{for }j\ne i.
\]
So we have
\begin{align}
V_0^i&=\epsilon\cdot\frac{1}{2}\left(n+1-\sum_{j}\frac{1}{x_j}-\sum_{j}\frac{x_i}{x_j}\right)+\sqrt{2\epsilon}\cdot\sqrt{x_i}z_0^i
\\
&=\frac{\epsilon}{2}\left(n+1-\sum_{j}\frac{1}{x_j}-\sum_{j}\frac{x_i}{x_j}\right)+\sqrt{2\epsilon x_i}z_0^i
\end{align}

Recall the exponential map from tangent space of Shahshahani manifold to the base manifold is
$\Exp_{\vec{x}}\left(\vec{v}\right)$,
in order to simplify the algorithm, we use orthogonal projection from the tangent space to the underlying manifold, instead of using the exact exponential map. In the Euclidean space, the manifold gradient descent for a function defined on $M\subset\mathbb{R}^d$ we mean the algorithm
\[
\vec{x}_{t+1}=\Retr_{\vec{x}_t}\left(-\epsilon\mathcal{P}_{T_{\vec{x}}}\nabla f(\vec{x}_t)\right)
\]
where $\mathcal{P}_{T_{\vec{x}}}\nabla f(\vec{x}_t)$ is the orthogonal projection of $\nabla f(\vec{x}_t)$ onto the tangent space $T_{\vec{x}_t}M$ with respect to the Euclidean metric on the ambient space $\mathbb{R}^d=T_{\vec{x}_t}\mathbb{R}^d$. To derive the Langevin MWU, we firstly generalize the manifold gradient on submanifold of Euclidean space to the case when the ambient space is a general Riemannian manifold $N$. Let $\grad_{N}f$ be the Riemannian gradient of $f$ on $N$, then the generalized gradient descent on the submanifold $M\subset N$ is of the following form:
\[
\vec{x}_{t+1}=\Retr_{\vec{x}_t}\left(-\epsilon\mathcal{P}_{T_{\vec{x}_t}}\grad_Nf(\vec{x}_t)\right)
\]
where the orthogonal projection of $\grad_Nf(\vec{x})$ onto the tangent space $T_{\vec{x}}M$ is based on the inner product $\langle,\rangle$ on the ambient tangent space $T_{\vec{x}}N$. For the case of Shahshahani manifold, let $N=\mathbb{R}^d_+$ and $M=\Delta_+^{d-1}$, and the orthogonal projection in the tangent space is with respect to the Shahshahani metric. Then for small $\epsilon>0$, the orthgonal projection of $-\epsilon\grad_Nf(\vec{x})$ onto $T_{\vec{x}}M$ is the vector obtained form the difference between the point $G(\vec{x})$ that is the normalization of $\vec{x}-\epsilon\grad f(\vec{x})$ onto the simplex and the initial point $\vec{x}$:
\begin{align}
-\vec{v}=-\mathcal{P}_{T_{\vec{x}}M}\vec{v}
=\mathcal{P}_{T_{\vec{x}}M}( \vec{v})
=G(\vec{x})-\vec{x}
\end{align}
where
\[
G(\vec{x})=\left(\frac{x_1- v_1}{1-\sum_jv_j},...,\frac{x_d- v_d}{1- \sum_jv_j}\right).
\]
Combined with the calculation of Brownian motion, we need to project the gradient and noise onto Shahshahani manifold,
\[
x_i\leftarrow\frac{x_i-\epsilon x_i\frac{\partial f}{\partial x_i}+V_0^i}{1-\epsilon\sum_{j}x_j\frac{\partial f}{\partial x_j}+\sum_j V_0^j}.
\]
So far we have completed the derivation of Langevin MWU when the reverse temperature $\beta=1$, it is trivial to scale $V_0^i$ with general $\beta$ to get the general form of Langevin MWU. The proof completes.
\end{proof}

\section{Missing Proofs}

\section{Proof of Lemma \ref{lemma:convergence1}}

\begin{proof}
Let $\vec{x},\vec{y}$ be two points on the Shahshahani manifold $M$ and $\vec{1}:=\frac{1}{n}(1,...,1)$. By the definition of geodesic, we claim that there exists a unique geodesic connecting $\vec{x}$ and $\vec{y}$. Especially this geodesic can be parametrized by length, i.e., there exists a parametrized curve $\gamma(t)$ with unit speed defined on $t\in[0,\ell]$ such that $\gamma(0)=\vec{x}$ and $\gamma(\ell)=\vec{y}$, where $\ell = d(\vec{x},\vec{y})$. Recall that $\langle,\rangle_{\gamma(t)}$ is the Riemannian metric at point $\gamma(t)$, for this pair of $\vec{x}$ and $\vec{y}$, we have the following
\begin{equation}
\begin{split}
f(\vec{y})-f(\vec{x})&=\int_0^{\ell}\langle\grad f(\gamma(t)),\gamma'(t)\rangle_{\gamma(t)}dt
\\
&\le \int_0^{\ell}\abs{\langle\grad f(\gamma(t)),\gamma'(t)\rangle_{\gamma(t)}}dt
\\
&\le\int_0^{\ell}\norm{\grad f(\gamma(t))}_{\gamma(t)}\cdot\norm{\gamma'(t)}_{\gamma(t)}dt.
\end{split}
\end{equation}
Suppose at $t=\xi\in[0,\ell]$, function $\norm{\grad f(\gamma(t))}_{\gamma(t)}\cdot\norm{\gamma'(t)}_{\gamma(t)}$ reaches its maximum value, then the above inequality can be simplified so that
\begin{equation}\label{eq:MVT}
f(\vec{y})-f(\vec{x})\le\norm{\grad f(\gamma(\xi))}_{\gamma(\xi)}\int_0^{\ell}\norm{\gamma'(\xi)}_{\gamma(\xi)}dt=\norm{\grad f(\gamma(\xi))}_{\gamma(\xi)}\cdot d(\vec{x},\vec{y}),
\end{equation}
where the last equality holds because $\gamma(t)$ is a unit speed geodesic. The assumption on $\grad f$ in the statement of this lemma gives that
\[
\norm{\grad f(\gamma(\xi))}_{\gamma(\xi)}\le c_1d(\vec{1},\gamma(\xi))+c_2.
\]
Since $\vec{1}$, $\vec{x}$ and $\vec{y}$ consist a geodesic triangle on $M$ and $\gamma(\xi)$ lies on the geodesic connecting $\vec{x}$ and $\vec{y}$, triangle inequality of the metric induced by Riemannian metric implies the following inequality,
\begin{equation}
\begin{split}
d(\vec{1},\vec{x})+d(\vec{x},\gamma(\xi))&\ge d(\vec{1},\gamma(\xi))
\\
d(\vec{1},\vec{y})+d(\vec{y},\gamma(\xi))&\ge d(\vec{1},\gamma(\xi))
\end{split}
\end{equation}
and then
\begin{equation}
\begin{split}
2d(\vec{1},\gamma(\xi))&\le d(\vec{1},\vec{x})+d(\vec{1},\vec{y})+d(\vec{x},\gamma(\xi))+d(\vec{y},\gamma(\xi))
\\
&=d(\vec{1},\vec{x})+d(\vec{1},\vec{y})+d(\vec{x},\vec{y})
\\
&\le 2\left(d(\vec{1},\vec{x})+d(\vec{1},\vec{y})\right).
\end{split}
\end{equation}
Combining with \ref{eq:MVT} we have
\begin{equation}
\begin{split}
f(\vec{y})-f(\vec{x})&\le \norm{\grad f(\gamma(\xi))}_{\gamma(\xi)}\cdot d(\vec{x},\vec{y})
\\
&\le\left(c_1d(\vec{1},\vec{x})+c_1d(\vec{1},\vec{y})+c_2\right)\cdot d(\vec{x},\vec{y}).
\end{split}
\end{equation}
Let $P$ be the coupling of $\mu$ and $\nu$, so that $W_2(\mu,\nu)^2=\mathbb{E}_Pd(X,Y)^2$ where $X\sim\mu$ and $Y\sim\nu$ respectively. Taking expectations of $f(\vec{x})-f(\vec{y})$, we have
\begin{equation}
\begin{split}
\int_M f\mu d\vol-\int_M f\nu d\vol=\mathbb{E}_P[f(X)-f(Y)]\le\mathbb{E}_P\abs{f(X)-f(Y)}.
\end{split}
\end{equation}

Moreover, 
\begin{equation}
\begin{split}
\mathbb{E}_P\abs{f(X)-f(Y)}&\le \mathbb{E}_P[(c_1d(\vec{1},\vec{x})+c_1d(\vec{1},\vec{y})+c_2)\cdot d(\vec{x},\vec{y})]
\\
&\le \sqrt{\mathbb{E}_P(c_1d(\vec{1},\vec{x})+c_1d(\vec{1},\vec{y})+c_2)^2}\sqrt{\mathbb{E}_Pd(\vec{x},\vec{y})^2}
\\
&\le \left(c_1\left(\mathbb{E}_{\mu}d(\vec{1},\vec{x})^2\right)^{\frac{1}{2}}+c_1\left(\mathbb{E}_{\nu}d(\vec{1},\vec{y})^2\right)^{\frac{1}{2}}+c_2\right)\cdot\sqrt{\mathbb{E}_Pd(\vec{x},\vec{y})^2}
\\
&=(c_1\sigma+c_2)W_2(\mu,\nu)
\end{split}
\end{equation}
We complete the proof by letting $c_1=\frac{M}{2}$ and $c_2=B$.
\end{proof}

\subsection{Proof of Lemma \ref{lemma:convergence2}}
\begin{proof}
Let $\nu(\vec{x})=\frac{e^{-\beta f(\vec{x})}}{\Lambda}$ denote the density function of the Gibbs measure with respect to the measure induced by the Shahshahani metric in simplex, where
\[
\Lambda = \int_Me^{-\beta f(\vec{x})}d\vol
\]
is the normalization constant known as the partition function. The differential entropy $h(\nu)$ is computed as follows,
\begin{equation}
\begin{split}
h(\nu)&=-\int_M\nu(\vec{x})\log\nu(x)d\vol
\\
&=-\int_M\frac{e^{-\beta f(\vec{x})}}{\Lambda}\log\frac{e^{-\beta f(\vec{x})}}{\Lambda}d\vol
\\
&=-\int_M\frac{e^{-\beta f(\vec{x})}}{\Lambda}(-\beta f(\vec{x})-\log\Lambda)d\vol
\\
&=\int_M\frac{e^{-\beta f(\vec{x})}}{\Lambda}(\beta f(\vec{x})+\log\Lambda)d\vol
\\
&=\int_M\frac{\beta f(\vec{x})e^{-\beta f(\vec{x})}}{\Lambda}d\vol+\int_M\frac{e^{-\beta f(\vec{x})}}{\Lambda}\log\Lambda d\vol
\\
&=\frac{1}{\Lambda}\int_M\beta f(\vec{x})e^{-\beta f(\vec{x})}d\vol+\log\Lambda
\end{split}
\end{equation} 
which implies
\[
\frac{h(\nu)}{\beta}=\frac{1}{\Lambda}\int_Mf(\vec{x})e^{-\beta f(\vec{x})}d\vol+\frac{\log\Lambda}{\beta}.
\]
Since $\nu(x)=\frac{e^{-\beta f(\vec{x})}}{\Lambda}$, we can furthermore obtain 
\[
\mathbb{E}_{\nu}f=\int_Mf(\vec{x})\nu(\vec{x})d\vol=\frac{h(\nu)}{\beta}-\frac{\log\Lambda}{\beta}
\]
Let $\vec{x}^*$ be any point that minimizes $f(\vec{x})$. Then $\grad f(\vec{x}^*)=0$. Since $f$ is assumed to be geodesically smooth, we have $f(\vec{x})-f(\vec{x}^*)\le\frac{M}{2}d(\vec{x},\vec{x}^*)^2$, the lower bound of $\log\Lambda$ can be obtained by the following calculation,
\begin{equation}
\begin{split}
\log\Lambda&=\log\int_Me^{-\beta f(\vec{x})}d\vol
\\
&=-\beta f(\vec{x}^*)+\log\int_Me^{\beta (f(\vec{x}^*)-f(\vec{x}))}d\vol
\\
&\ge -\beta f(\vec{x}^*)+\log\int_Me^{-\frac{\beta M}{2}d(\vec{x},\vec{x}^*)^2}d\vol.
\end{split}
\end{equation}
Note that showing the boundedness of the integral $\int_M\exp\left(-\frac{\beta M}{2}d(\vec{x},\vec{x}^*)^2\right)d\vol$ can be reduced to the boundedness of $\int_M\exp\left(-\frac{\beta M}{2}d(\vec{1}_n,\vec{x})^2\right)d\vol$ by a translation on $M$. The translation is defined with parallel transport. Suppose $\gamma(t)$ is the geodesic connecting $\vec{1}$ and an arbitrary point $\vec{y}$, $\vec{v}$ is the vector in the tangent space at $\vec{1}$ such that $\Exp_{\vec{1}}(\vec{v})=\vec{x}^*$. Let $\Gamma_{\vec{1}}^{\vec{y}}\vec{v}$ be the parallel transport of $\vec{v}$ along $\gamma(t)$, and define the image of $\vec{y}$ to be the point $\Exp_{\vec{y}}(\Gamma_{\vec{1}}^{\vec{y}}\vec{v})$.  Lemma \ref{lemma:distance} provides us the value of $\int_M\exp\left(-\frac{\beta M}{2}d(\vec{1},\vec{x})^2\right)d\vol$ by letting $c=\frac{\beta M}{2}$, i.e., 
\begin{equation}
\begin{split}
&\int_M\exp\left(-\frac{\beta M}{2}d(\vec{1},\vec{x})^2\right)d\vol
\\
&=\alpha_{n-1}\frac{1}{2}\left(\frac{2}{\beta M}\right)^{\frac{n-1}{2}}\Gamma\left(\frac{n-1}{2}\right)
\\
&-\alpha_{n-1}\frac{\tau(R)}{6(n-1)}\frac{1}{2}\left(\frac{2}{\beta M}\right)^{\frac{n+1}{2}}\Gamma\left(\frac{n+1}{2}\right)
\\
&+\frac{\alpha_{n-1}}{360(n-1)(n+1)}(-3\norm{R}^2+8\norm{\rho(R)}^2+5\tau(R)^2-18\Delta R)\frac{1}{2}\left(\frac{2}{\beta M}\right)^{\frac{n+3}{2}}\Gamma\left(\frac{n+3}{2}\right)
\\
&+\alpha_{n-1}\left(\frac{2}{\beta M}\right)^{\frac{n+5}{2}}\Gamma\left(\frac{n+5}{2}\right)O(1),
\end{split}
\end{equation}
By denoting $\text{poly}\left(\frac{1}{\beta}\right)$ the right hand side for short, we have
\[
\log\Lambda\ge-\beta f(\vec{x}^*)+\log\left(\text{poly}\left(\frac{1}{\beta}\right)\right).
\]
Dividing both side by $\beta$ and rearranging, we have
\[
\frac{\log \Lambda}{\beta}\ge -f(\vec{x}^*)+\frac{1}{\beta}\log\left(\text{poly}\left(\frac{1}{\beta}\right)\right)
\]
and 
\[
-f(\vec{x}^*)\le\frac{\log\Lambda}{\beta}-\frac{1}{\beta}\log\left(\text{poly}\left(\frac{1}{\beta}\right)\right)=\frac{\log\Lambda}{\beta}+\frac{1}{\beta}\log\left(\text{poly}\left(\frac{1}{\beta}\right)^{-1}\right).
\]
Combined with $\mathbb{E}_{\nu}f=\frac{h(\nu)}{\beta}-\frac{\log\Lambda}{\beta}$ that has been proven before, we have 

\[
\mathbb{E}_{\nu}f-f(\vec{x}^*)\le \frac{h(\nu)}{\beta}+\frac{1}{\beta}\log\left(\text{poly}\left(\frac{1}{\beta}\right)^{-1}\right)\le\frac{K}{\beta}+\frac{1}{\beta}\log\left(\text{poly}\left(\frac{1}{\beta}\right)^{-1}\right).
\]
\end{proof}

\begin{lemma}\label{lemma:distance}
Let $M$ be the Shahshahani manifold, $\vec{1}:=(\frac{1}{n},...,\frac{1}{n})^{\top}$. Then for any $c>0$, the integral of $e^{-cd(\vec{1},\vec{x})^2}$ over $M$ with respect to the volume form induced by Shahshahani metric is bounded, i.e.,
\[
\int_Me^{-cd(\vec{1},\vec{x})^2}d\vol
\]
is bounded.
\end{lemma}

\begin{proof}
The distance between $\vec{1}_n$ and any other point $\vec{x}$ on $M$ can be computed by the exponential map on Shahshahani manifold. Recall that the exponential map on $M$ at $\vec{x}$ is given by
\[
\Exp_{\vec{x}}(\vec{v})=\left(\frac{x_1e^{v_1}}{S},...,\frac{x_ne^{v_n}}{S}\right)
\]
where $S=\sum_jx_je^{v_j}$. Consider the tangent space at $\vec{1}_n$, i.e., $T_{\vec{1}_n}M$, for any $\vec{x}\in M$, there exists a unique $\vec{v}\in T_{\vec{1}_n}M$ such that $\Exp_{\vec{1}_n}(\vec{v})=\vec{x}$. This can be done by solving equations given as follows, 
\[
\left(\frac{e^{v_1}}{S},...,\frac{e^{v_n}}{S}\right)=(x_1,...,x_n).
\]
We have $\frac{e^{\sum v_i}}{S^n}=\prod_{i=1}^n x_i$. Thus $S=\left(\frac{1}{\prod_{i=1}^nx_i}\right)^{\frac{1}{n}}$, and then
$
v_i=\ln x_iS$, for all $i\in[n]$. 

Furthermore, if the exponential map is defined on $T_{\frac{1}{n}\vec{1}_n}M$, $S=\frac{1}{n}$, and then $v_i=\ln x_i-\ln n$ for all $i\in[n]$. The distance $d(\frac{1}{n}\vec{1}_n,\vec{x})$ between $\frac{1}{n}\vec{1}_n$ and $\vec{x}$ can be obtained by evaluating the Shahshahani length of the vector $\vec{v}$ at $\frac{1}{n}\vec{1}_n$, where $\vec{v}$ is the one satisfying $\Exp_{\frac{1}{n}\vec{1}_n}(\vec{v})=\vec{x}$. Since we already have the expression of each component of $\vec{v}$ with respect to the exponential map on the tangent space at $\frac{1}{n}\vec{1}_n$, the square of the distance $d(\frac{1}{n}\vec{1}_n,\vec{x})$ can be computed as follows,
\begin{equation}
d\left(\frac{1}{n}\vec{1}_n,\vec{x}\right)^2=\norm{\vec{v}}_{\frac{1}{n}\vec{1}_n}^2=(v_1,...,v_2)\left(
\begin{array}{ccc}
n&&
\\
&\ddots&
\\
&&n
\end{array}\right)
\left(
\begin{array}{c}
v_1
\\
\vdots
\\
v_n
\end{array}
\right)=n\sum_{i=1}^nv_i^2=n\sum_{i=1}^n(\ln x_i-\ln n)^2.
\end{equation}
Therefore, the integral $\int_Me^{-cd(\frac{1}{n}\vec{1}_n,\vec{x})^2}d\vol$ can be explicitly written as
\[
\int_M\exp\left(-cd\left(\vec{1},\vec{x}\right)^2\right)d\vol=\int_M\exp\left(-cn\sum_{i=1}^n(\ln x_i-\ln n)^2\right)d\vol.
\]
Since the volume form $d\vol$ is induced from the Shahshahani metric on $M$, which is not compact. In fact, each geodesic of infinite length can be embedded into $M$, i.e., for any $\vec{v}\in T_{\frac{1}{n}\vec{1}_n}M$, its image $\left(\frac{x_1e^{tv_1}}{S},...,\frac{x_ne^{tv_n}}{S}\right)$ converges asymptotically to a point on the boundary of simplex as $t\rightarrow \infty$. The function $\exp\left(-cd\left(\vec{1},\vec{x}\right)^2\right)$ can be integrated along any geodesic starting from $\vec{1}$. In order to estimate the integral presented in the beginning, we consider $\vec{x}$ is obtained by mapping $t\vec{v}\in T_{\vec{1}}M$ to $M$, where $\vec{v}$ has Shahshahani norm of $1$ at $\frac{1}{n}\vec{1}_n$. Then the integral can be written in terms of $t$ and $\vec{v}$, which is a polar coordinate system in Shahshahani manifold. Therefore, the integral is computed by an integration over the geodesic from $\vec{1}$ followed by an integration over a sphere $\vec{v}\in S^{n-2}$. Let's elaborate it as follows by denoting $\gamma(t)$ the geodesic with initial velocity $\vec{v}$.

\begin{equation}
\begin{split}
\int_M\exp\left(-cd\left(\vec{1},\vec{x}\right)^2\right)d\vol&=\int_0^{\infty}\left(\int_{S_{n-2}(t)}e^{-ct^2d(\vec{1},\vec{v})^2}dS_{n-2}(t)\right)\norm{\gamma'(t)}_{\gamma(t)}dt
\end{split}
\end{equation}
Note that the integral inside can be computed as follows,
\[
\int_{S_{n-2}(t)}e^{-ct^2d(\vec{1},\vec{v})^2}dS_{n-2}(t)=e^{-ct^2d(\vec{1},\vec{v})^2}\text{Volume}(S_{n-2}(t)),
\]
and we need some notations to estimate the volume of $S_{n-2}(t)$. Let $R$ be Riemannian curvature tensor on $n-1$ dimensional manifold $M$, from \cite{Alfred} we can write
\[
\tau(R)=\sum_{i=1}^{n-1}R_{ii}, \ \ \ \norm{R}^2=\sum_{i,j,k,l=1}^{n-1}R_{ijkl}^2,
\]
\[
\norm{\rho(R)}^2=\sum_{i,j=1}^{n-1}R_{ij}^2,\ \ \ \Delta R=\text{Laplacian of }R=\sum_{i=1}^n\nabla_{ii}^2\tau(R).
\]

Denote
\[
\alpha_{n-1}=\frac{2\Gamma(\frac{1}{2})^{n-1}}{\Gamma(\frac{n-1}{2})},
\]
and then
\begin{equation}
\begin{split}
\text{Volume}(S_{n-2}(t))&=\alpha_{n-1}t^{n-2}\left(1-\frac{\tau(R)}{6(n-1)}t^2+\right.
\\
&\left.\frac{1}{360(n-1)(n+1)}\left(-3\norm{R}^2+8\norm{\rho(R)}^2+5\tau(R)^2-18\Delta R\right)t^4+O(t^6)\right).
\end{split}
\end{equation}
We are now ready to compute the integral
\begin{equation}
\begin{split}
\int_Me^{-cd(\vec{1},\vec{x})^2}d\vol&=\int_0^{\infty}e^{-ct^2}\text{Volume}(S_{n-2}(t))dt
\\
&=\int_0^{\infty}e^{-ct^2}\alpha_{n-1}t^{n-2}dt
\\
&-\int_0^{\infty}e^{-ct^2}\alpha_{n-1}\frac{\tau(R)}{6(n-1)}t^ndt
\\
&+\int_0^{\infty}\frac{\alpha_{n-1}}{360(n-1)(n+1)}\left(-3\norm{R}^2+8\norm{\rho(R)}^2+5\tau(R)^2-18\Delta R\right)t^{n+2}e^{-ct^2}dt
\\
&+\int_0^{\infty}\alpha_{n-1}t^{n-2}O(t^6)e^{-ct^2}dt
\end{split}
\end{equation}
It is immediate to calculate more general integral in the form of 
\[
\int_0^{\infty}e^{-ct^2}t^kdt
\]
by writing $r=ct^2$, the above integral equals
\begin{equation}
\begin{split}
\frac{1}{2}\left(\frac{1}{c}\right)^{\frac{k+1}{2}}\int_0^{\infty}e^{-r}r^{\frac{k-1}{2}}dr&=\frac{1}{2}\left(\frac{1}{c}\right)^{\frac{k+1}{2}}\int_0^{\infty}e^{-r}r^{\frac{k+1}{2}-1}dr
\\
&=\frac{1}{2}\left(\frac{1}{c}\right)^{\frac{k+1}{2}}\Gamma\left(\frac{k+1}{2}\right).
\end{split}
\end{equation}
Thus we have
\begin{equation}
\begin{split}
\int_Me^{-cd(\vec{1},\vec{x})^2}d\vol&=\alpha_{n-1}\frac{1}{2}\left(\frac{1}{c}\right)^{\frac{n-1}{2}}\Gamma\left(\frac{n-1}{2}\right)
\\
&-\alpha_{n-1}\frac{\tau(R)}{6(n-1)}\frac{1}{2}\left(\frac{1}{c}\right)^{\frac{n+1}{2}}\Gamma\left(\frac{n+1}{2}\right)
\\
&+\frac{\alpha_{n-1}}{360(n-1)(n+1)}(-3\norm{R}^2+8\norm{\rho(R)}^2+5\tau(R)^2-18\Delta R)\frac{1}{2}\left(\frac{1}{c}\right)^{\frac{n+3}{2}}\Gamma\left(\frac{n+3}{2}\right)
\\
&+\alpha_{n-1}\left(\frac{1}{c}\right)^{\frac{n+5}{2}}\Gamma\left(\frac{n+5}{2}\right)O(1).
\end{split}
\end{equation}

The proof completes.

\end{proof}

The next proposition shows that one can obtain a complete analogy of the conditions in \cite{RRT17}. 

\begin{proposition}\label{lemma:bound of f}
Suppose function $f$ satisfies Assumptions 1-5. It holds that
\[
\norm{\grad f(\vec{x})}_{\vec{x}}\le \frac{M}{2}d(\vec{1},\vec{x})+B
\]
and 
\[
\frac{m}{2}d(\vec{1},\vec{x})^2-\frac{b}{2}\ln 3\le f(\vec{x})\le A+\frac{M}{2}d(\vec{1},\vec{x})^2+Bd(\vec{1},\vec{x}).
\]
\end{proposition}

\begin{proof}
Direct estimate gives the bound of gradient 
\[
\norm{\grad f(\vec{x})}_{\vec{x}}\le\frac{M}{2}d(\vec{1},\vec{x})+B.
\]
Suppose $\gamma(t)$ is the geodesic connecting $\vec{1}$ and $\vec{x}$. Using fundamental theorem of calculus along geodesic, we have
\[
f(\vec{x})-f(\vec{1})=\int_0^1\langle\grad f(\gamma(t)),\gamma'(t)\rangle_{\gamma(t)}dt
\]
where $\langle\cdot,\cdot\rangle_{\gamma(t)}$ is the Riemannian metric at $\gamma(t)$. By assumption that $f(\vec{1})\le A$, we have
\begin{equation}
\begin{split}
f(\vec{x})&=f(\vec{1})+\int_0^1\langle\grad f(\gamma(t)),\gamma'(t)\rangle_{\gamma(t)}dt
\\
&\le A+\int_0^1\langle\grad f(\gamma(t)),\gamma'(t)\rangle_{\gamma(t)}dt
\\
&\le A+\int_0^1\norm{\grad f(\gamma(t))}_{\gamma(t)}\cdot\norm{\gamma'(t)}_{\gamma(t)}dt
\\
&\le A+\int_0^1\left(\frac{M}{2}d(\vec{1},\gamma(t))+B\right)\norm{\gamma'(t)}_{\gamma(t)}dt
\\
&\le A+\int_0^1\left(\frac{M}{2}d(\vec{1},\vec{x})+B\right)\norm{\gamma'(t)}_{\gamma(t)}dt
\\
&=A+\left(\frac{M}{2}d(\vec{1},\vec{x})+B\right)\int_0^1\norm{\gamma'(t)}_{\gamma(t)}dt
\\
&=A+\left(\frac{M}{2}d(\vec{1},\vec{x})+B\right)d(\vec{1},\vec{x})
\\
&=A+\frac{M}{2}d(\vec{1},\vec{x})^2+Bd(\vec{1},\vec{x}).
\end{split}
\end{equation}
To show the other inequality, we consider the geodesic $\gamma(t)$ connecting $\vec{1}$ and $\vec{x}$, suppose $c\in[0,1]$ is a point on $\gamma(t)$. For that we can use more geometric features of geodesic, we assume the geodesic to be of constant speed, i.e., $\norm{\gamma'(t)}_{\gamma(t)}=\norm{\vec{v}}_{\vec{1}}=d(\vec{1},\vec{x})$, and then the length of the geodesic connecting $\vec{1}$ and $\vec{x}$ equals $\norm{\vec{v}}_{\vec{1}}=d(\vec{1},\vec{x})$. We have that
\[
f(\vec{x})=f(\gamma(c))+\int_c^{1}\langle f(\gamma(t)),\gamma'(t)\rangle_{\gamma(t)}dt
\]
where $f(\gamma(c))\ge 0$. By the dissipative assumption, i.e.,
\[
\langle\grad f(\gamma(t)),\frac{\gamma'(t)}{\norm{\gamma'(t)}_{\gamma(t)}}d(\vec{1},\gamma(t))\rangle_{\gamma(t)}\ge md(\vec{1},\gamma(t))^2-b
\]
Then we can have the following estimate,
\begin{equation}
\begin{split}
f(\vec{x})&\ge\int_c^{1}\langle\grad f(\gamma(t)),\gamma'(t)\rangle_{\gamma(t)}dt
\\
&=\int_c^1\langle\grad f(\gamma(t)),\frac{\gamma'(t)}{\norm{\gamma(t)}_{\gamma(t)}}d(\vec{1},\gamma(t))\rangle_{\gamma(t)}\frac{\norm{\gamma'(t)}_{\gamma(t)}}{d(\vec{1},\gamma(t))}dt
\\
&\ge\int_c^1\left(md(\vec{1},\gamma(t))^2-b\right)\frac{\norm{\gamma'(t)}_{\gamma(t)}}{d(\vec{1},\gamma(t))}
\\
&=\int_c^1\left(mt^2\norm{\vec{v}}_{\vec{1}}^2-b\right)\frac{\norm{\vec{v}}_{\vec{1}}}{t\norm{\vec{v}}_{\vec{1}}}dt
\\
&=\int_c^1\left(mt\norm{\vec{v}}_{\vec{1}}^2-\frac{b}{t}\right)dt
\\
&=\frac{m(1-c^2)}{2}\norm{\vec{v}}_{\vec{1}}^2+b\ln c
\\
&=\frac{m(1-c^2)}{2}d(\vec{1},\vec{x})^2+b\ln c.
\end{split}
\end{equation}
Taking $c=\frac{1}{\sqrt{3}}$, we have 
\[
\frac{m}{2}d(\vec{1},\vec{x})^2-\frac{b}{2}\ln 3\le f(\vec{x})\le A+\frac{M}{2}d(\vec{1},\vec{x})^2+Bd(\vec{1},\vec{x}).
\]
The proof completes.

\end{proof}

\end{document}